\documentclass[a4paper,oneside]{amsart}
\usepackage{geometry}

\usepackage{mathtext}
\usepackage[T1]{fontenc}
\usepackage{color}
\usepackage{amsmath}
\usepackage{amsfonts,amssymb,amsthm}
\usepackage[pdftex]{graphicx}
\usepackage{comment}
\usepackage{float}
\usepackage{todonotes}
\usepackage{microtype}

\usepackage[colorlinks=false, linktocpage=true]{hyperref}

\DeclareRobustCommand{\SkipTocEntry}[5]{}

\usepackage{tikz}
\usetikzlibrary{decorations.markings}
\tikzset{middlearrow/.style={
		decoration={markings,
			mark= at position 0.5 with {\arrow{#1}} ,
		},
		postaction={decorate}
	}
}
\tikzset{firstthirdarrow/.style={
		decoration={markings,
			mark= at position 0.33 with {\arrow{#1}} ,
		},
		postaction={decorate}
	}
}
\tikzset{secondthirdarrow/.style={
		decoration={markings,
			mark= at position 0.66 with {\arrow{#1}} ,
		},
		postaction={decorate}
	}
}
\usetikzlibrary{cd}

\usepackage{enumerate}

\usepackage{cite}

\DeclareMathOperator{\sgn}{sgn}

\DeclareMathOperator{\Aut}{Aut}

\DeclareMathOperator{\Hom}{Hom}

\DeclareMathOperator{\Id}{Id}

\DeclareMathOperator{\Span}{Span}

\DeclareMathOperator{\Sp}{Sp}
\DeclareMathOperator{\GL}{GL}
\DeclareMathOperator{\SL}{SL}

\DeclareMathOperator{\Rep}{Rep}

\DeclareMathOperator{\spp}{\mathfrak{sp}}

\DeclareMathOperator{\Fix}{Fix}

\DeclareMathOperator{\Is}{Is}

\DeclareMathOperator{\TPA}{TPA}
\DeclareMathOperator{\Path}{Path}

\usepackage{color}
\definecolor{NoteColor}{rgb}{1,0,0}

\newcommand{\R}{\mathbb R}
\newcommand{\RR}{\mathbb R}
\renewcommand{\SS}{\mathbb S}
\newcommand{\CC}{\mathbb C}
\newcommand{\HH}{\mathbb H}
\newcommand{\PP}{\mathbb P}

\newcommand{\Z}{\mathbb Z}
\newcommand{\ZZ}{\mathbb Z}

\newcommand{\T}{\mathcal T}
\newcommand{\Ome}[1]
{\begin{pmatrix}
		0 & #1\\
		-#1 & 0
\end{pmatrix}}

\theoremstyle{plain}
\newtheorem{teo}{Theorem}[section]
\newtheorem{theorem}[teo]{Theorem}
\newtheorem{cor}[teo]{Corollary}

\newtheorem{lem}[teo]{Lemma}
\newtheorem{lemma}[teo]{Lemma}
\newtheorem{prop}[teo]{Proposition}
\newtheorem{proposition}[teo]{Proposition}

\theoremstyle{definition}
\newtheorem{df}[teo]{Definition}
\newtheorem{definition}[teo]{Definition}

\theoremstyle{remark}
\newtheorem{rem}[teo]{Remark}
\newtheorem{remark}[teo]{Remark}

\usepackage{xparse}
\DeclareDocumentEnvironment{proofeqn}{}{\[}{\qedhere\]}

\newcommand{\bs}{\setminus}
\newcommand{\defin}{\emph}
\renewcommand{\bar}{\overline}

\begin{document}
	
	\title{On partial abelianization of framed local systems}
	
	\author[C. Kineider]{Clarence Kineider}
	
	\author[E. Rogozinnikov]{Eugen Rogozinnikov}
	
	\keywords{Framed local systems, decorated local systems, spectral networks, abelianization}
	
	\address{Clarence Kineider \\ Institut de Recherche Math\'ematique Avanc\'ee, Universit\'e de Strasbourg, Strasbourg, France}
	\email{clarence.kineider@math.unistra.fr}
	
	\address{Eugen Rogozinnikov \\ Institut de Recherche Math\'ematique Avanc\'ee, Universit\'e de Strasbourg, Strasbourg, France}
	\email{erogozinnikov@gmail.com}
	
	\begin{abstract}
	   D.~Gaiotto, G.~W.~Moore and A.~Neitzke introduced spectral networks to understand the framed $G$-local systems over punctured surfaces for $G$ a split Lie group via a procedure called abelianization. We generalize this construction to groups $G$ of the form  $\GL_2(A)$, where $A$ is a unital associative ring, and to some of its subgroups. This relies on a precise analysis of the degree 2 ramified coverings associated with spectral networks and triangulations and on a matrix reinterpretation of their path lifting rules; along the way we provide another proof of the Laurent phenomenon brought to light by A.~Berenstein and V.~Retakh. The partial abelianization enables us to gives parametrizations of the moduli spaces of decorated  $G$-local systems and of framed $G$-local systems over punctured surfaces. For $(A, \sigma)$ a Hermitian involutive $\R$-algebra the group $G=\Sp_2(A, \sigma)$ is a classical Hermitian Lie group of tube type, and we are able to identify and parametrize the moduli space of maximal framed $G$-local systems.
	\end{abstract}

	\thanks{E.R.\;thanks the Labex IRMIA of the Universit\'e de Strasbourg for support during the preparation of this article. C.K.\;thanks the ENS Rennes for support during the preparation of this article.}
	
	\maketitle
	
	\tableofcontents
	
	\section{Introduction}
	The theory of spectral networks was developed by D.\;Gaiotto, G.\,W.\;Moore and A.\;Neitzke \cite{gmn13, gmn14, hn16, moore} during their research on supersymmetric quantum field theory. However, the mathematical objects arising from this work proved to have an independent mathematical interest. The abelianization using spectral networks can be applied to the study of the geometry of the character varieties of surface groups into complex Lie groups and split real Lie groups~\cite{Alessandrini_vid}.
	
	Spectral networks can be seen as graphs on a degree $n$ ramified covering of a given surface. For $n>2$ a generic spectral network is an infinite graph that is dense on the surface, however finite spectral networks exist for every $n\geq 2$. The case $n=2$ is the simplest one, but the abelianization procedure in this case can be only applied to a very restricted class of split Lie groups of rank $1$ (e.g. $\SL_2(\RR)$, $\SL_2(\CC)$).
	
	The main purpose of this paper is to generalize the abelianization procedure described by D.\;Gaiotto, G.\,W.\;Moore and A.\;Neitzke to Lie groups $G$ that can be seen as $\GL_2(A)$ or some subgroups of $\GL_2(A)$ for some unital associative not necessarily commutative $\R$-algebra $A$. Although, such groups are not always split of rank $1$, the abelianization procedure can be partially applied for these groups. In this way, we can understand the structure of the moduli space of decorated and framed $G$-local systems over punctured surfaces.
	
	We now describe our results in more detail.
	
	Let~$S$ be a surface without boundary of negative Euler characteristic $\chi(S)$ with punctures (we refer to Section~\ref{sec:top_data} for the wider generality that can be allowed for~$S$, for example disks with marked points on the boundary). A \defin{decorated surface} is a surface as above together with a choice of a simple smooth loop (decorating loop) in a neighborhood of every puncture.
	
	Let $G$ be a subgroup of $\GL_2(A)$ for some unital associative not necessarily commutative $\R$-algebra $A$. A \defin{twisted $G$-local system} on $S$ is a local system on the unit tangent bundle $T'S$ of $S$ with the holonomy around the fiber of $T'S\to S$ equal to $-1$.
	
	Further, we consider $A^2$ (seen as the space of column vectors) as a right $A$-module. The group $G$ acts on $A^2$ by the left multiplication. A \defin{framing} of a twisted $G$-local system is a choice of a~parallel line $A$-subbundle in a neighborhood of every puncture. A \defin{decoration} of a twisted $G$-local system is a choice of a parallel regular section along every decorating loop. For a precise definition of those notions see Section~\ref{sec:ab_framed}. A twisted $G$-local system together with a framing (or decoration) is called a \defin{framed (resp. decorated) twisted $G$-local system}.
	
	Notice that a parallel regular section along a decorating loop always induces a parallel line $A$-subbundle in a neighborhood of the corresponding puncture. Hence, a decorated twisted $G$-local system always admits a natural framing.
	
	Fixing an ideal triangulation~$\mathcal{T}$ of $S$, we consider a subspace of the space of twisted (framed or decorated) $G$-local system, that are \defin{transverse} with respect to $\mathcal T$ (or just $\mathcal T$-transverse). Following \cite{gmn13,hn16}, we introduce the ramified covering $\Sigma\to S$ adapted to the triangulation $\mathcal T$ and the spectral network on $\Sigma$ as a graph that satisfy some axioms (for more detail we refer to Section~\ref{sec:spectral_networks}).
	
	For $\mathcal T$-transverse framed twisted $G$-local system we describe the twisted abelianization procedure using spectral networks adapted to the triangulation $\mathcal T$. The result of this procedure is a twisted $A^\times$-local system on $\Sigma$. We also show the converse, i.e. that for every twisted $A^\times$-local system on $\Sigma$ there exist a unique twisted $\mathcal T$-transverse framed $G$-local system for $G=\GL_2(A)$ (non-abelianization). We describe the abelianization and non-abelianization procedures by defining a path-lifting map from the \defin{twisted path algebra} (see Section~\ref{sec:tw_ab}) of $S$ to the twisted path algebra of $\Sigma$, and we show this map is homotopy-invariant.
	
	Using this construction, we define non-commutative $\mathcal{A}$-coordinates on the space of decorated twisted $G$-local systems, and using the path-lifting map we show that these coordinates provide a geometric realization of the non-commutative algebra introduced in \cite{BR}. This allows us to give a geometrical proof of the non-commutative Laurent phenomenon, first shown in \cite{BR}.
	
	Further, we use this abelianization procedure to understand the topology of $\mathcal T$-transverse (framed and decorated) twisted $G$-local systems:
	
	\begin{teo}
	    Let $S$ be a punctured orientable surface of negative Euler characteristic $\chi(S)$ without boundary. Then the moduli space of framed (twisted) $\GL_2(A)$-local systems on $S$ that are transverse to a fixed triangulation $\T$ is homeomorphic to the moduli space of  (twisted) $A^\times$-local systems on $\Sigma$ which is homeomorphic to $(A^\times)^{1-4\chi(S)}/A^\times$ where $A^\times$ acts diagonally by conjugation on $(A^\times)^{1-4\chi(S)}$.
		
		The moduli space of decorated twisted unipotent $\GL_2(A)$-local systems on $S$ that are transverse to a fixed triangulation $\T$ is homeomorphic to the product of the moduli space of twisted $A^\times$-local systems on $\bar\Sigma$ and $(A^\times)^{p}$ where $p$ is the number of punctures of $S$.
 
		
	\end{teo}
	
	Finally, we introduce involutive algebras $(A,\sigma)$, i.e. unital, associative $\R$-algebras with the $\R$-linear map $\sigma\colon A\to A$ such that $\sigma(ab)=\sigma(b)\sigma(a)$ for all $a,b\in A$ and $\sigma^2=\Id$. Over involutive algebras the symplectic group can be defined as follows: $\Sp_2(A,\sigma):=\{g\in\GL_2(A)\mid \sigma(g)^t\omega g=\omega\}$ where $\omega=\begin{pmatrix}0 & 1 \\ -1 & 0\end{pmatrix}$. These groups were studied in~\cite{ABRRW}, and they are of particular interest for higher rank Teichm\"uller theory: For a special class of involutive algebras $(A,\sigma)$ called Hermitian algebras, the groups $\Sp_2(A,\sigma)$ are Hermitian of tube type. This gives rise to so-called maximal $\Sp_2(A,\sigma)$-local systems on $S$ and maximal representations of the fundamental group of $S$ into $\Sp_2(A,\sigma)$. Maximal local systems and maximal representations were introduced and studied in~\cite{BIW,BILW,Strubel}. They provide examples of so-called Higher Teichm\"uller spaces, i.e. subspaces of the character variety $\Rep(\pi_1(S),\Sp_2(A,\sigma))=\Hom(\pi_1(S),\Sp_2(A,\sigma))/\Sp_2(A,\sigma)$ that consist entirely of discrete and faithful representations. The topology of spaces of maximal representation for closed surfaces was studied in~\cite{Gothen, Guichard_Wienhard_InvaMaxi, BGPG, AC}, partly using the theory of Higgs bundles. In~\cite{AGRW}, the spaces of framed and decorated maximal representations into the real symplectic group $\Sp(2n,\R)$ are parametrized using a non-commutative analog of the Fock--Goncharov parametrization~\cite{FG} and the topology of them is studied.
	
	We introduce (framed and decorated) twisted $\Sp_2(A,\sigma)$-local system and describe the topology of the moduli space of $\mathcal T$-transverse framed twisted $\Sp_2(A,\sigma)$-local systems:
	
	\begin{teo}\label{intro:sympectic_local_systems}
		Using the same notations as in the previous theorem, the moduli space of framed (twisted) $\Sp_2(A,\sigma)$-local systems on $S$ that are transverse to a fixed triangulation $\T$ is homeomorphic to:
		$$\left(((A^\sigma)^\times)^{-2\chi(S)}\times (A^\times)^{1-\chi(S)}\right)/A^\times$$
		where $A^\sigma=\Fix_A(\sigma)$, $A^\times$ acts componentwisely by conjugation on $(A^\times)^{1-\chi(S)}$ and by congruence on $((A^\sigma)^\times)^{-2\chi(S)}$.
  
	\end{teo}
	
	For Hermitian $A$, we also introduce maximal (framed and decorated) twisted $\Sp_2(A,\sigma)$-local system and describe the topology of the moduli space of maximal framed twisted symplectic local systems:
	
	\begin{teo}
		If $A$ is Hermitian, then the moduli space of framed (twisted) maximal $\Sp_2(A,\sigma)$-local systems on $S$ is homeomorphic to:
		$$\left((A^\sigma_+)^{-2\chi(S)}\times (A^\times)^{1-\chi(S)}\right)/A^\times$$
		where $A^\sigma_+=\{a^2\mid a\in (A^\sigma)^\times\}$, $A^\times$ acts componentwisely by conjugation on $(A^\times)^{1-\chi(S)}$ and by congruence on $(A^\sigma_+)^{-2\chi(S)}$.
  
	\end{teo}
	This provides a new proof of the result of~\cite{AGRW,GRW,R-Thesis}.
	
    \addtocontents{toc}{\SkipTocEntry}
    \subsection*{Structure of the paper:}
	In Section~\ref{sec:top_data} we introduce the topological and combinatorial data needed for the abelianization process, such as ramified coverings and a special class of graphs on them called spectral networks. We define the path-lifting map related to the spectral network. In Section~\ref{sec:ab_framed} we describe the partial abelianization and non-abelianization processes for framed twisted $\GL_2(A)$-local systems. In Section~\ref{sec:ab_decorated} we apply this construction to decorated twisted $\GL_2(A)$-local systems, and relate it to the non-commutative algebra introduced in \cite{BR}. We also describe the topology of the moduli space of both framed and decorated twisted $\GL_2(A)$-local systems that are transverse with respect to a fixed ideal triangulation. In Section~\ref{sec:symplectic_case} we specify this construction for $\Sp_2(A,\sigma)$-local systems, and describe the topology of the moduli space of maximal framed twisted symplectic local systems.

    \addtocontents{toc}{\SkipTocEntry}
    \subsection*{Acknowledgements}
    We thank Daniele~Alessandrini, Olivier~Guichard and Anna~Wienhard for helpful and interesting discussions about some aspects of this article.
	
	\section{Topological and combinatorial data}\label{sec:top_data}
	
	\subsection{Punctured surface}
	
	Let $\bar S$ be a compact orientable smooth surface of finite type with or without boundary. Let $P$ be a nonempty finite subset of $\bar S$ such that on every boundary component of $\bar S$ there is at least one element of $P$. We define $S:=\bar S\bs P$.  Elements of $P$ are called \defin{punctures} of $S$. Sometimes we will distinguish between elements of $P$ that lie in the interior of $\bar S$ -- \defin{internal punctures} and that lie on the boundary -- \defin{external punctures}. Surfaces that can be obtained in this way are called \defin{punctured surfaces}, with the exception of the (closed) disk with one or two punctures on the boundary and the sphere with one or two punctures. Every punctured surface can be equipped with a complete hyperbolic structure of finite volume with totally geodesic boundary. For every such hyperbolic structure, all the internal punctures are cusps and all boundary curves are (infinite) geodesics. Once equipped with a hyperbolic structure as above, the universal covering $S'$ of $S$ can be seen as a closed convex subset of the hyperbolic plane $\HH^2$ with totally geodesic boundary, which is invariant under the natural action of $\pi_1(S)$ on $\HH^2$ by the holonomy representation. Punctures of $S$ are lifted to points of the ideal boundary of $\HH^2$ which we call punctures of $S'$ and denote their set by $P'\subseteq \partial_\infty S'\subseteq \partial_\infty\HH^2$. Notice, if $\bar S$ does not have boundary, then $S'$ is the entire $\HH^2$.
	
	An \defin{ideal triangulation} of $S$ is a triangulation with oriented edges of $\bar S$ whose set of vertices agrees with $P$, such that, if $\gamma$ is an edge of the triangulation, then the opposite edge $\bar\gamma$ is also an edge of this triangulation. We always consider edges of an ideal triangulation as homotopy classes of oriented paths (relative to their endpoints) connecting points in $P$. Connected components of the compliment on $S$ to all edges of an ideal triangulation $\mathcal T$ are called \defin{faces} or \defin{triangles} of $\mathcal T$. Every edge belongs to the boundary of one or two triangles. In the first case, an edge is called \defin{external}, in the second -- \defin{internal}. Any ideal triangulation of $S$ can be represented by an ideal geodesic triangulation as soon as a hyperbolic structure as above on $S$ is chosen.
	
	\subsection{Ramified covering}
	
	Let $\mathcal T$ be an ideal triangulation of $S$. We can endow $\bar S$ with a Euclidean structure with conical points by choosing for each triangle $T$ of $\mathcal{T}$ an orientation preserving diffeomorphism $\varphi_T : \mathsf{T}\to T$ where $\mathsf{T}$ is the Euclidean triangle in $\RR^2 = \CC$ with vertices $1$, $j = e^{\frac{2i\pi}{3}}$ and $j^2$. Then for each gluing of two triangles (not necessarily distinct) in $S$, glue the corresponding Euclidean triangles with the composition of a rotation and a translation.
	The conical points of this structure are exactly the points in $P$, meaning that this structure once restricted to $S$ is smooth.
	Let $B = \left\lbrace \varphi_T(0)~|~T \text{ triangle of }\mathcal{T} \right\rbrace \subset S$. There is one point of $B$ in the interior of each triangle of $\mathcal{T}$. With this data, we can construct a two-fold branched covering $\pi\colon\bar\Sigma\to \bar S$ such that the branched points are precisely elements of $B$ and $\bar\Sigma$ has a Euclidean structure.
	Let $\mathsf{H}$ be the Euclidean hexagon with vertices the sixth roots of unity in $\CC$. Then the map $z\mapsto z^2$ is a ramified covering from $\mathsf{H}$ to $\mathsf{T}$ that has exactly one ramification of order 2 at the point 0. Then take as many copies of $\mathsf{H}$ as there are triangles in $\mathcal{T}$ and for each gluing of two triangles (not necessarily distinct) in $S$, glue the corresponding Euclidean hexagons on both edges that are mapped to the glued edge in $S$ with rotation and a translation.
	
	\begin{figure}[h!]
		\includegraphics[scale=0.5]{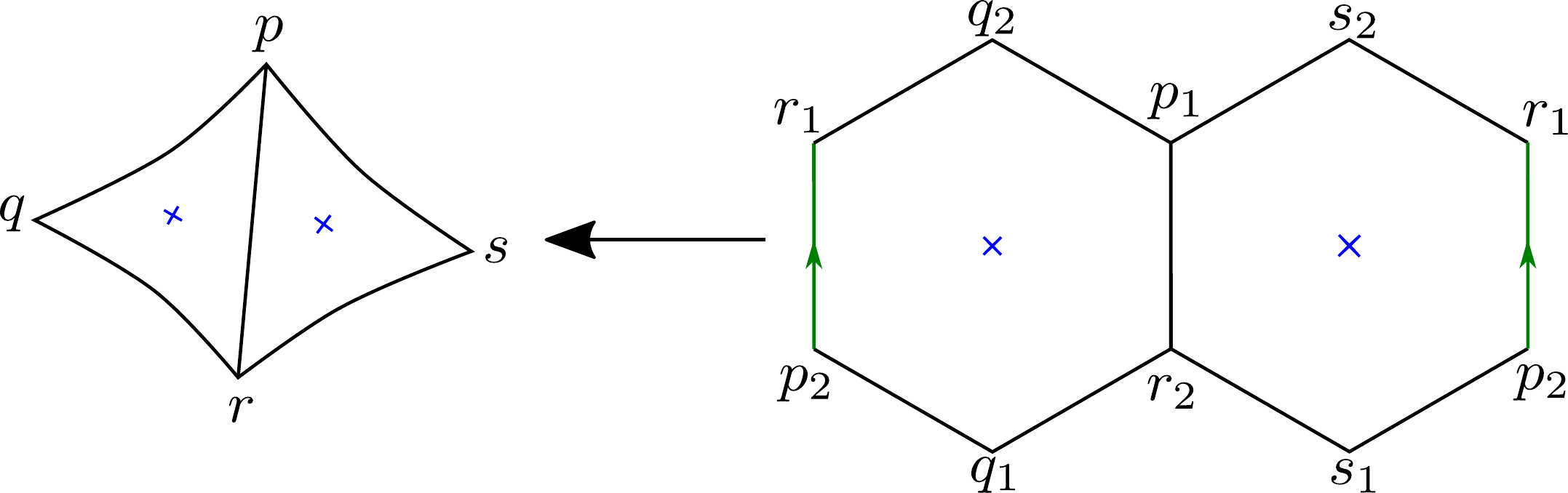}
		\caption{The ramified two-fold covering of two glued triangles. The preimages of $p$ are $p_1$ and $p_2$, same for $q,r,s$. The branched points are the blue crosses. The two outer edges with an arrow are glued according to arrow orientation.}
	\end{figure}
	
	This defines a two-fold ramified covering $\pi : \bar{\Sigma}\to\bar{S}$ with ramification points at $B$, and the conical points of $\bar\Sigma$ are a subset of $\pi^{-1}(P)$. This means the map $\pi$ restricted to $\Sigma = \bar{\Sigma}\setminus\pi^{-1}(P)$ is a smooth two-fold branched covering from $\Sigma$ to $S$, with simple ramifications on points of $B$.
	The lift $\T^*:=\pi^{-1}(\mathcal T)$ of $\mathcal T$ to $\Sigma$ induces a hexagonal tiling of $\Sigma$ such that in every hexagon there is exactly one element of $\pi^{-1}(B)$.
	
	\begin{rem}
		Obviously, triangles of $\mathcal T$ are in 1:1-correspondence with elements of $B$, and hexagons of $\pi^{-1}(\mathcal T)$ are also in 1:1-correspondence with elements of $B$.
	\end{rem}
	
	\begin{definition}
		For a smooth manifold $X$, denote $T'X$ the spherical quotient of $TX$, i.e. $T'X = T^pX/\RR^*_+$ where $T^pX$ is the punctured tangent bundle of $X$ and the group $\RR^*_+$ acts fiberwise by multiplication. The space $T'X$ is then a sphere bundle over $X$, and we will write an element of $T'X$ as an ordered pair $(x,v)$ with $x\in X$ and $v$ a non-zero vector in $T_xX$, identified with the half-line it spans. With a slight abuse of terminology, we will call this sphere bundle the \defin{unit tangent bundle} of $X$.
	\end{definition}

	\begin{rem}
		Since the map $\pi$ is a local diffeomorphism on $\Sigma\setminus \pi^{-1}(B)$, it induces the tangent (differential) map $d\pi\colon T(\Sigma\setminus \pi^{-1}(B))\to T(S\setminus B)$ that factorizes to unit tangent bundles $T'(\Sigma\setminus \pi^{-1}(B))\to T(S\setminus B)'$. In order to simplify the notation, we will sometime write $\pi\colon T\Sigma\to TS$ and $\pi\colon T'\Sigma\to T'S$ instead of $d\pi$.
	\end{rem}
	
	\begin{rem}
		The unit tangent bundle of $\mathsf{H}$ is canonically identified to $\mathsf{H}\times \SS^1$ as $\mathsf{H}$ is a subset of $\R^2$. With this identification, the preimages by $d\pi$ of $(x,v)\in T'S$ are of the form $(x_1,v')$ and $(x_2,-v')$ where $x_1$ and $x_2$ are the preimages of $x$ by $\pi$.
	\end{rem}
	
	The following proposition describe the topology of the ramified covering $\Sigma$:
	
	\begin{proposition}\label{prop:topology_Sigma}
		Let $\bar S$ be a compact orientable surface with $k\geq 0$ boundary components $C_1,\dots,C_k$ and let $P$ be a finite set of points of $\bar S$ such that for all $i\in\left\lbrace 1,\dots,k\right\rbrace $, $n_i = \# (C_i\cap P) >0$. Let $k_e$ (resp. $k_o$) be the number of components of $\partial\bar S$ with an even (resp. odd) number of punctures, such that $k=k_e+k_o$. Let $p=\# (P\setminus\partial\bar S) $, let $g$ be the genus of $\bar S$ and let $S=\bar S\setminus P$. Then the two-fold ramified covering $\Sigma = \bar\Sigma\setminus\pi^{-1}(P)$ of $S$ is a surface such that:
		\begin{itemize}
			\item $\bar\Sigma$ is a compact orientable surface of genus
			$$g'=\frac{1}{2}\left(2p+2k_e+3k_o+8g-6 +\sum_{i=1}^k n_i\right),$$
			\item for each of the $k_e$ boundary components $C$ of $\bar S$ with even number $n$ of punctures, $\pi^{-1}(C)$ is the union of two distinct boundary components in $\bar\Sigma$, each with $n$ punctures,
			\item for each of the $k_o$ boundary components $C$ of $\bar S$ with odd number $n$ of punctures, $\pi^{-1}(C)$ is one boundary component in $\bar\Sigma$ with $2n$ punctures,
			\item $\Sigma$ has $2p$ internal punctures.
		\end{itemize}
	\end{proposition}
	
	\begin{proof}
		First, note that the genus $g'$ of $\Sigma$ is an integer because $3k_o+\sum n_i$ is always even. It is clear from the construction that $\bar\Sigma$ is compact and orientable, and that $\Sigma$ has $2p$ internal punctures. To compute the number of boundary components of $\bar\Sigma$, we will glue to each boundary of $\bar S$ a disk with the corresponding number of puncture on the boundary to get a surface $\hat{S}$ with no boundary, only internal punctures. Since a disk with one (resp. two) puncture on the boundary does not admit an ideal triangulation, we glue a disk with one (resp. two) puncture on the boundary and one internal puncture instead.
		In the corresponding ramified covering $\hat{\Sigma}$ of $\hat{S}$, we then remove the lifts of the interior of the glued disks to obtain $\Sigma$. The result follows from the following lemma:
		\begin{lemma}
			If $S$ is a closed disk with $n\geq 3$ punctures on the boundary, $\bar\Sigma$ has either one boundary component with $2n$ punctures if $n$ is odd or two boundary components with $n$ punctures each if $n$ is even. If $S$ is a disk with one internal puncture and one puncture on the boundary, $\bar\Sigma$ has one boundary component with two punctures. If $S$ is a disk with one internal puncture and two punctures on the boundary, $\bar\Sigma$ has two boundary components with two punctures each.
		\end{lemma}
		\begin{proof}
		    The two cases with an internal puncture can be computed individually. Let $S$ be a disk with $n\geq 3$ punctures on the boundary.
		    Let $\mathcal{T}$ be a triangulation of $S$ and $\Sigma$ the corresponding ramified covering. Let $\gamma$ be a loop homotopic to the boundary of the disk going around all the $n-2$ branched points in $S$. Let $x$ be the base point of $\gamma$, and $x_1,x_2$ the lifts of $x$ to $\Sigma$. Let $\tilde\gamma$ the lift of $\gamma$ starting at $x_1$. If $\tilde\gamma$ is a loop then there are two lifts of the boundary of $\bar S$ to $\bar \Sigma$, and if $\tilde\gamma$ is a path from $x_1$ to $x_2$ then the lift of the boundary of $\bar S$ is connected in $\bar\Sigma$. The loop $\gamma$ is homotopic to the concatenations of loops $\gamma_1,\dots,\gamma_{\lfloor\frac{n-2}{2}\rfloor},\gamma'$ based at $x$ such that each $\gamma_i$ goes around two branched points in $S$ and $\gamma'$ is either trivial if $n-2$ is even or goes around one branched point if $n-2$ is odd. Then $\tilde\gamma$ is the concatenation of the lifts $\tilde\gamma_1,\dots,\tilde\gamma_{\lfloor\frac{n-2}{2}\rfloor},\tilde\gamma'$. Since the $\tilde\gamma_i$ are loops based at $p_1$ and $\tilde\gamma'$ is either trivial or a path from $p_1$ to $p_2$ (depending on the parity of $n$), we get the result.
		\end{proof}
		The Euler characteristic of $\bar S$ is $$\chi(\bar S) = 2-2g-k = 2-2g-k_o-k_e$$
		and the Euler characteristic of $\bar{\Sigma}$ is
		$$\chi(\bar\Sigma) = 2-2g'-k_o-2k_e.$$
		The number of branched points is the same as the number of triangles in $\mathcal{T}$, which is $-2\chi(\bar S)+2p+\sum n_i$. Riemann-Hurwitz formula gives us:
		\begin{align*}
			\chi(\bar\Sigma)=2-2g'-k_o-2k_e &= 2\chi(\bar S)-\left(-2\chi(\bar S)+2p+\sum_{i=1}^k n_i\right)\\
			&= 4\chi(\bar S)-2p-\sum_{i=1}^k n_i\\
			&= 8-8g-2p-4k_o-4k_e -\sum_{i=1}^k n_i
		\end{align*}
		We can then solve for $g'$ to get the result.
	\end{proof}

	\begin{remark}
		In particular, the topology of $\Sigma$ does not depend on the triangulation $\mathcal{T}$.
	\end{remark}
	
	We denote by $\theta\colon\Sigma\to\Sigma$ the covering involution. The following result is a direct consequence of the above proposition.
	
	\begin{cor}\label{counting}
		The fundamental group $\pi_1(\Sigma)$ is a free group of rank $$1-\chi(\bar\Sigma)+2p=1-4\chi(\bar S)+4p+\sum n_i.$$
		Let $b\in \Sigma$ be a ramification point of the covering $\pi\colon\Sigma\to S$. Let $\alpha_1,\dots,\alpha_s\colon [0,1]\to S$ be free generators of the fundamental group $\pi_1(S,\pi(b))$ that do not pass through other ramification points. The fundamental group $\pi_1(\Sigma,b)$ is the free group freely generated by the following collection of loops on $\Sigma$:
		\begin{enumerate}
			\item For every generator $\alpha_i$, there are two closed lifts $\gamma_i^1$ and $\gamma_i^2=\theta\circ\gamma_i^1$ on $\Sigma$ based at $b$ (in total $2-2\chi(\bar S)+2p$ curves);
			\item For every ramification point $b'\neq b$ in $\Sigma$, we fix a simple segment on $S$ connecting $\pi(b)$ and $\pi(b')$ and take the lift of this segment on $\Sigma$. It is a closed loop $\xi$ based at $b$ (in total $-2\chi(\bar S)+2p-1+\sum n_i$ curves).
		\end{enumerate}
		
		The fundamental group $\pi_1(T'\Sigma,\tilde b)$ where $\tilde b\in T'\Sigma$ is a lift of $b$ to $T'\Sigma$ is generated by lifts of curves described above and the curve going once around the fiber of $T'\Sigma\to\Sigma$ at $\tilde b$.
		
	\end{cor}
	
	\subsection{Spectral network}\label{sec:spectral_networks}
	Let $\mathcal{T}$ be an ideal triangulation of $S$ and $\pi : \Sigma\to S$ be the corresponding $2$-to-$1$ ramified covering. A (small) \defin{spectral network} associated with this ramified covering is a set $\mathcal{W}$ of paths $\left[ -1,1\right] \to \bar\Sigma$ (called \defin{rays}) satisfying:
	\begin{itemize}
		\item for all $\alpha\in\mathcal{W}$, $\alpha(-1),\alpha(1)\in \pi^{-1}(P)$, $\alpha(0)\in B$ and if $t\notin\left\lbrace -1,0,1\right\rbrace $, then $\alpha(t)\notin P\cup B$
		\item for all $\alpha\in\mathcal{W}$ and for all $t\in [-1,1]$, $\pi(\alpha(t)) = \pi(\alpha(-t))$
		\item for all $b\in B$, there are exactly 3 rays $\alpha_1,\alpha_2,\alpha_3\in\mathcal{W}$ passing through $b$, and locally around $b$ the rays look like this:
		\begin{figure}[h!]
			\includegraphics[scale=0.5]{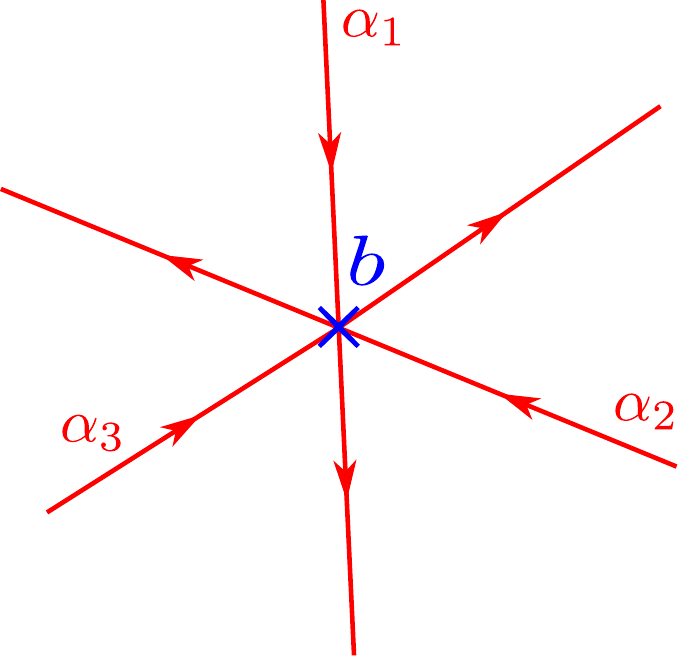}
		\end{figure}
		\item for all $\alpha \neq \alpha', \alpha(]-1,0[)$ -- which we call the \defin{past} of $\alpha$ -- does not intersect $\alpha'(]0,1[)$ -- which we call the \defin{future} of $\alpha'$.
	\end{itemize}
	
	\begin{remark}
		We can omit the last condition in rank 2 spectral networks (i.e. associated to a two-fold covering) as we will construct spectral networks without intersections in $\Sigma$.
	\end{remark}
	
	We can construct a spectral network associated with a triangulation $\mathcal{T}$ in the following way: call the points of $\pi^{-1}(P)$ that are on third roots of unity (for the Euclidean structure) \defin{sinks} and all other points of $\pi^{-1}(P)$ \defin{sources}. This way, every point of $P$ have two preimages, one source and one sink.
	For each hexagon $\mathsf{H}$ of $\pi^{-1}(\mathcal{T})$, the three rays going through the branch point in $\mathsf{H}$ are the three Euclidean segments going from the source to the sink for each of the three puncture in $\mathsf{T}$.
	
	\begin{figure}[h!]\label{SN-triangle}
		\includegraphics[scale=0.5]{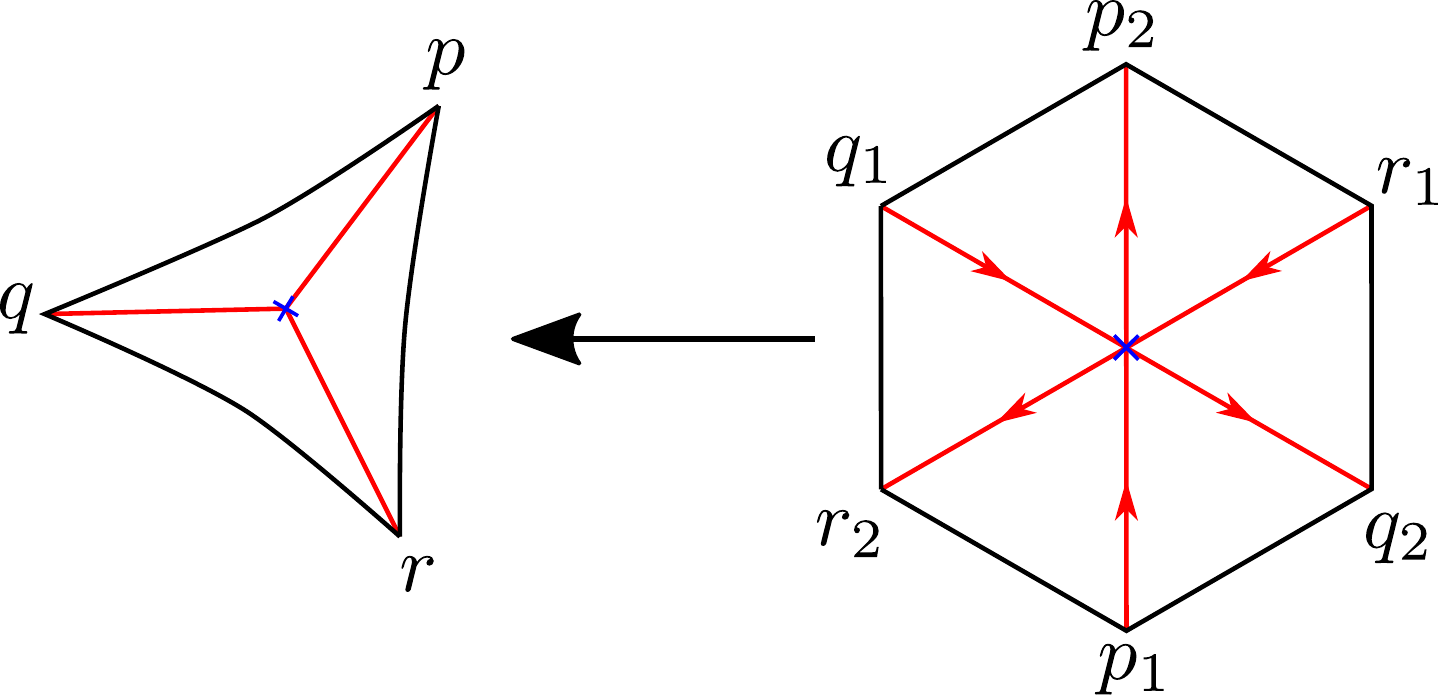}
		\caption{Picture of the spectral network on each triangle of the triangulation. Here the sources are $p_2,q_2,r_2$ and the sinks are $p_1,q_1,r_1$.}
	\end{figure}
	
	We fix a spectral network $\mathcal W$ on $\Sigma$ adapted to the covering $\Sigma\to S$ and to the ideal triangulation $\mathcal T$. The complement of all lines of $\mathcal W$ on $\Sigma$ is a collection of simply connected regions called \defin{cells}. These are either quadrilaterals bounded by four lines of $\mathcal W$ or triangles bounded by two lines of $\mathcal W$ and one boundary component of $S$. The closure of every cell in $\bar S$ contains precisely two punctures.
	
	Let now $p_0\colon S'\to S$ be the universal covering of $S$. There exist a branched two-fold covering $\pi'\colon\Sigma'\to S'$ and an (infinite but in general not universal) covering $p_1\colon \Sigma'\to \Sigma$ such that the following diagram commutes:
	$$\begin{tikzcd}
		\Sigma' \arrow{r}{\pi'} \arrow[swap]{d}{p_1} & S' \arrow{d}{p_0} \\%
		\Sigma \arrow{r}{\pi}& S
	\end{tikzcd}$$
	The ideal hexagonal tiling of $\Sigma$ lifts to an ideal hexagonal tiling of $\Sigma'$. We call the set of ends of edges of this tiling \defin{the ideal boundary of $\Sigma'$}. In fact, the ideal boundary of $\Sigma'$ does not depend on the choice of the tiling.
	
	The map $\pi'$ can be continuously extended to the ideal boundary of $\Sigma'$. Therefore, we can talk about images and preimages of punctures under $\pi'$.
	
	We lift the triangulation $\mathcal T$ on $S$, the corresponding hexagonal tiling on $\Sigma$ and the spectral network $\mathcal W$ to these coverings.
	
	Now we are working on the universal covering $S'$. For every puncture $p$ of $S'$, we consider the union of all cells that have this puncture in their (ideal) boundaries, take the closure of this union in $S'$ and then take the interior of this set. This is an open contractible set of $S'$, we denote it by $U_p$ and call the \defin{standard neighborhood} of the puncture $p\in S'$. Let $p_1$ and $p_2$ be two lifts of the puncture $p$ of $S'$ under $\pi'$. The lift of $U_p$ to $\Sigma'$ consists of two connected components $U_{p_1}$ and $U_{p_2}$ every of which projects homeomorphically to $U_p$, i.e. $U_p$ is evenly covered by $U_{p_1}$ and $U_{p_2}$. We will call $U_{p_1}$ (resp. $U_{p_2}$) the \defin{standard neighborhood} of $p_1$ (resp. $p_2$).
	
	Notice that two standard neighborhoods either do not intersect or intersect in a cell of the (lifted) spectral network $\mathcal W$. More precisely, two standard neighborhoods $U_{p}$ and $U_{q}$ intersect if and only if the punctures $p$ and $q$ of $S'$ (resp. of $\Sigma'$) are connected by an edge of $p_0^{-1}(\mathcal T)$ in $S'$ (resp. by an edge of $p_1^{-1}(\T^*)$ in $\Sigma'$).

	\subsection{Peripheral decoration}
	
	To consider framed and decorated local systems, we will need the following additional data on the surface $S$:  for every internal puncture $p\in P$ we fix a neighborhood $S_p\subset S$ of $p\in P$ that is diffeomorphic to a punctured disk. For every external puncture $p\in P$, we choose a neighborhood $S_p\subset S$ of $p$ that is diffeomorphic to a punctured half-disk. We also assume that all $S_p$ are so small that they are pairwise disjoint, and their union does not contain points of $B$. In this case, every $S_p$ is evenly covered by $\Sigma_{p_1}$ and $\Sigma_{p_2}$ where $\{p_1,p_2\}=\pi^{-1}(p)$ and $\Sigma_{p_1}$ and $\Sigma_{p_2}$ are the two connected components of $\pi^{-1}(S_p)$.
	
	Further, for every internal puncture $p\in P$ we fix a simple smooth loop $\beta_p\colon [0,1]\to S_p$ around $p$ such that $\dot\beta_p(0)=\dot\beta_p(1)$, oriented such that $p$ is on the right of $\beta_p$ according to the orientation of $S$. For every external puncture $p\in P$, we chose a simple smooth path $\beta_p\colon ([0,1],\{0,1\})\to (S_p, S_p\cap\partial S)$ connecting two boundary connected components separated by $p$, once again with orientation given by the one on $S$.

	In both cases, up to isotopy there is only one such $\beta_p$. Since all $\beta_p$ are smooth, we can lift them to the $T'S$ namely to the curve $[\beta_p(t),\dot\beta_p(t)]\in T'S$, $t\in[0,1]$. We denote this lift by $T'\beta_p\colon [0,1]\to T'S$. Notice that for every internal puncture $p$, the lift $T'\beta_p$ is always a loop.
	
	If for every $p\in P$ a curve $\beta_p$ as above is chosen, then we say that the surface $S$ is \defin{decorated}, and the collection $\mathcal D=\{\beta_p\mid p\in P\}$ is called a \defin{decoration} of $S$. 
	
	If a hyperbolic structure as above on $S$ is chosen, then every $\beta_p$ can be represented by projections of small enough horocycles around some $p'\in p_0^{-1}(p)$ under the universal covering map $p_0\colon S'\to S$.
	
	Let $\mathcal{T}$ be an ideal triangulation of $S$. We can assume the arcs of the triangulation are smooth, and if the surface is decorated, we will further assume that any arc of the triangulation intersects only once the peripheral curves associated to its endpoints and do so with matching derivative, i.e. an arc $\gamma\in\mathcal{T}$ from $p\in P$ to $q\in P$ satisfies $\dot\gamma(t_0)=\dot\beta_p(t_0')$ and $\dot\gamma(t_1)=\dot\beta_q(t_1')$ where $t_0,t_1,t_0',t_1'$ are such that $\gamma(t_0) = \beta_p(t_0')$ and $\gamma(t_1)=\beta_q(t_1')$. This is the same as assuming that every arc from $p$ to $q$ of the lift $\mathcal{T}'$ of $\mathcal{T}$ to $T'S$ intersects the lifts $T'\beta_p$ and $T'\beta_q$. This can be done by bending the arcs of $\mathcal{T}$ in a neighborhood of their intersections with the peripheral curves. Let $I_\mathcal{T}(S)\subset T'S$ be the set of intersection points between $\mathcal{T}'$ and the lifted decoration curves $T'\beta_p$. This means that now each edge of the triangulation $\mathcal{T}$ is endowed with two special points (one for each extremity) lying on the peripheral curves associated to its endpoints.
	For every edge $\gamma\in\mathcal{T}$ of the triangulation, let $\tau_\gamma$ be the path in $T'S$ with extremities in $I_\mathcal{T}$ obtained by restricting $\gamma$ to the part in between the two special points on it. Note that since this is applied to all the oriented arcs of the triangulation, the chosen representative for $\tau_\gamma$ and $\tau_{\bar\gamma}$ are such that $T'(\tau_\gamma.\tau_{\bar\gamma})$ is homotopic to a lace that loop once around the fiber $T'S\to S$ (see Figure~\ref{fig:bending}).
	\begin{figure}[h!]\label{fig:bending}
		\includegraphics[scale = 0.3]{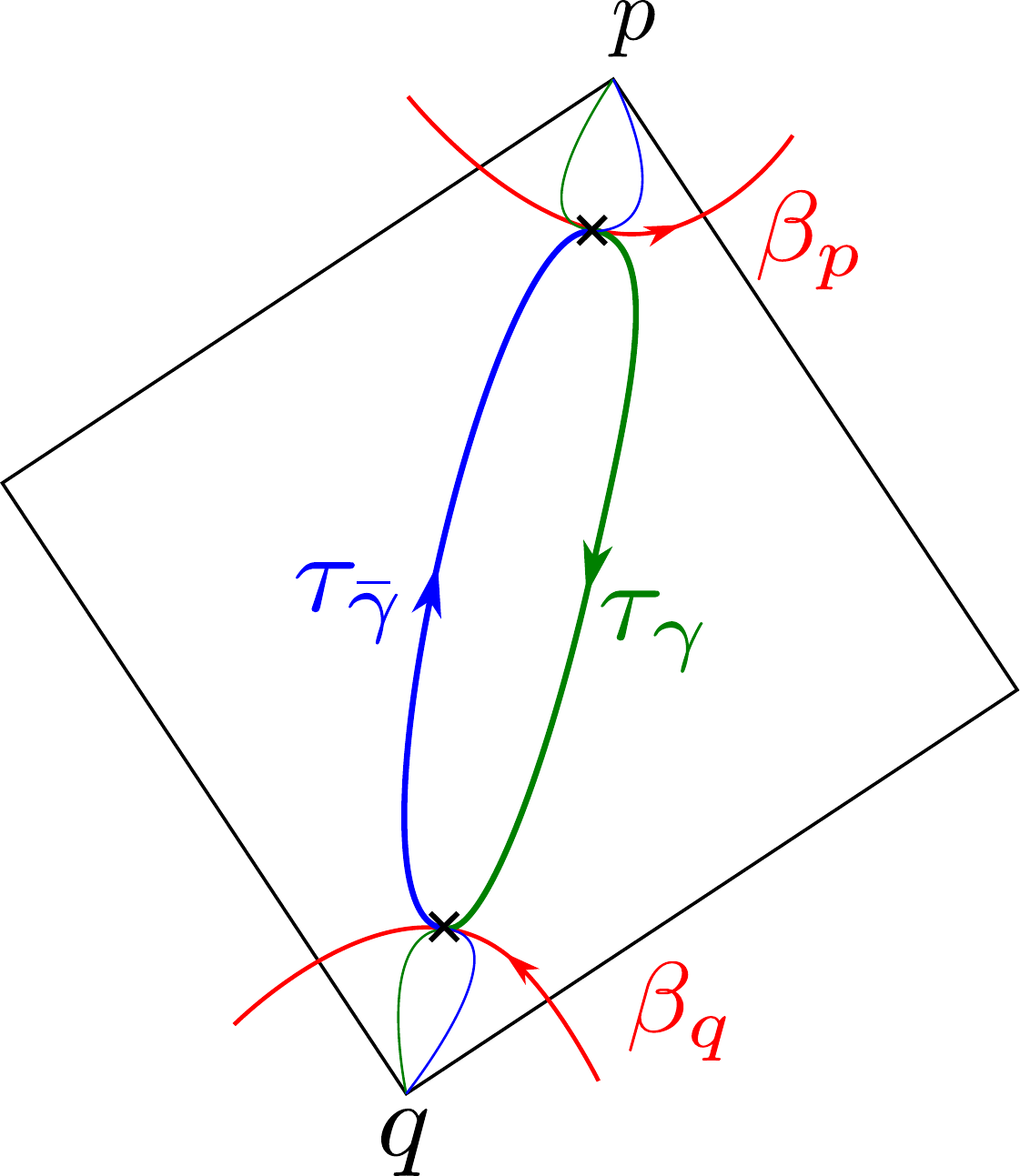}
		\caption{The bending of an edge of the triangulation. In red are the peripheral decoration, in blue and green are the oriented edges of the triangulation, the crosses are the points in $I_\mathcal{T}$ and the thicker part of the edges in between the peripheral curves are the paths $\tau_\gamma$ and $\tau_{\bar\gamma}$}
	\end{figure}
	We also apply the same construction in $\Sigma$ to equip each edge of the hexagonal tiling $\mathcal{T}^*$  with two special points, and denote $I_\mathcal{T^*}(\Sigma)$ the set of all special points in $T'\Sigma$. Note that both $I_\mathcal{T^*}(\Sigma)$ and $I_\mathcal{T}(S)$ are finite sets, and that $\pi : T'\Sigma\to T'S$ is 2:1 from $I_\mathcal{T^*}(\Sigma)$ to $I_\mathcal{T}(S)$.
	
	Since the points in $I_{\mathcal{T}^*}(\Sigma)$ lie on peripheral curves associated with punctures, they inherit the source/sink naming from the puncture.
	
	\subsection{The spectral network map}\label{sec:tw_ab}
	
	Lifting paths to a ramified covering is not homotopy invariant: a contractible loop around a branch point $b\in B$ is lifted as two paths on $\Sigma$ that are not loops, thus not homotopic to the lift of the trivial loop. The goal of this section is to construct a path-lifting map $SN$, which depends on the spectral network $\mathcal{W}$, from paths on $T'S$ to paths on $T'\Sigma$ such that $SN$ is well-defined on homotopy classes. 
	
	We will use the symbol $\approx$ to represent homotopy (with fixed extremities) of paths.
	
	Let $S$ be a punctured surface, $\mathcal{T}$ an ideal triangulation of $S$ and $\pi : \Sigma\to S$ the two-fold branched covering constructed previously. Let $\mathcal{W}$ be the spectral network adapted to this covering constructed before.
	Every path $\alpha : ]-1,1[\to \Sigma$ of $\mathcal{W}$ is smooth since it is a straight line for the Euclidean structure. We can thus lift the paths of $\mathcal{W}$ to $T'\alpha :\begin{array}{rcl}
		]-1,1[&\to &T'\Sigma \\
		t&\mapsto &(\alpha(t),\dot\alpha(t))
	\end{array}$. We will also call this set of paths in $T'\Sigma$ a spectral network and denote it $T'\mathcal{W}$. 
	
	Let $H$ be a hexagonal tile of $\Sigma$. Note that the Euclidean structure on $\Sigma$ allows us to identify $T'H$ with $H\times \SS^1$. In the following, a path $\gamma$ on $T'H\simeq H\times \SS^1$ will be written as a couple $(x,v)$ where $x$ is the projection of $\gamma$ on $H\subset\Sigma$ and $v$ is the projection of $\gamma$ on $\SS^1$. Note that $\SS^1$ has a natural orientation given by the one on $\Sigma$. For all $\theta\in\SS^1$, define $s^+_\theta$ to be the (homotopy class of the) path in $\SS^1$ going from $\theta$ to $-\theta$ following the orientation of $\SS^1$, and $s^-_\theta$ going from $\theta$ to $-\theta$ in the opposite direction.
	For a path $v$ on $\SS^1$, we will denote $-v$ the image of $v$ under the involution $\theta\mapsto -\theta$. The path $-v$ goes from $-v(0)$ to $-v(1)$. In particular, we have $s^-_\theta = -\overline{s^+_\theta}$ and $(-s^\pm_\theta).s^\pm_\theta = \delta^\pm_\theta$ where $\delta^\pm_\theta : t\mapsto \theta \pm 2i\pi t$. When the context is clear, we will omit the subscript describing the starting point of the paths $s^\pm$ and $\delta^\pm$. The paths $\delta^\pm$ satisfy $\delta^- = \overline{\delta^+}$ and if $v$ is a path on $\SS^1$ from $\theta_1$ to $\theta_2$, we have $\delta^\pm_{\theta_2}.v \approx v.\delta^\pm_{\theta_1}$.
	
	The Euclidean structure on $\Sigma$ also define a flat connection $\nabla$ on $T\Sigma$ given by the restriction of the standard flat connection on $\RR^2$.
	Since it is a bilinear map on the sections of $T\Sigma$ (denoted $\Gamma(T\Sigma)$), this connection induces a flat connection (which we also call $\nabla$) on the unit tangent bundle
	\[\nabla : \Gamma(T'\Sigma)\times \Gamma(T'\Sigma)\to \Gamma(T'\Sigma).\]

	\begin{definition}
		Let $X$ be a topological space. The path algebra of $X$ (denoted $\Z[\Path(X)]$) is the free $\Z$-algebra generated by homotopy classes of paths $[0,1]\to X$, with the product given by concatenation of paths:
		if $\gamma_1(0)\neq \gamma_2(1)$ then $\gamma_1.\gamma_2 = 0$ and if $\gamma_1(0)= \gamma_2(1)$ then $\gamma_1.\gamma_2$ is the path obtained by going through $\gamma_2$ then $\gamma_1$.
	\end{definition}
	
	Now let $X$ be a smooth surface. Define the \defin{twisted path algebra} of $X$ as
	\[ \TPA(X)  = \Z[\Path(T'X)]/\mathcal{I}\]
	where $\mathcal{I}$ is the two-sided ideal generated by the elements $e_{x,\theta}+\delta_{x,\theta}$ for $(x,\theta)\in T'X$, with
	\[e_{x,\theta} : \begin{array}{rcl}
		[0,1]&\to & T'X \\
		t&\mapsto& (x,\theta)
	\end{array} \text{ and } \delta_{x,\theta} : \begin{array}{rcl}
		[0,1]&\to&T'X\\
		t&\mapsto& (x,\theta+2\pi t)
	\end{array}.\]
	
	\begin{remark}\label{rk:TPA_subalg}
		Given any non-empty subset $E\subset T'X$, the subset \[\left\lbrace \gamma_1+\dots+\gamma_r+ \mathcal{I}~|~\text{for all } 1\leq i\leq r,\text{ endpoints of } \gamma_i \text{ are in }E \right\rbrace\subset \TPA(X)\] is a subring of $\TPA(X)$ because composition of paths preserves the set of endpoints. We will denote $\TPA_E(X)$ this subring.
	\end{remark}
	
	Let $x$ be a path on $S$ intersecting only once (and not at its endpoints) the spectral network $\mathcal{W}$ and not going through a branch point. Let $\alpha\in\mathcal{W}$ be the path such that $\pi(\alpha)$ intersects $x$. The two standard lifts $x_1$ and $x_2$ of $x$ to $\Sigma$ each intersect once $\alpha$, one of them intersecting the past of $\alpha$ and the other intersecting the future of $\alpha$. Suppose $x_1$ is the one intersecting the past of $\alpha$. We can then define a new path $x'$ on $\Sigma$ as the concatenation of 5 paths $x'_1,\dots,x'_5$ defined as follows:
	\begin{itemize}
		\item $x'_1$ is the part of $x_1$ from its starting point to the intersection point with $\alpha$
		\item $x'_2$ is the part of $\alpha$ from the intersection with $x_1$ to the branch point
		\item $x'_3$ is a constant path at the branch point (it will be useful in the next paragraph when we will consider the lifted spectral network $T'\mathcal{W}$)
		\item $x'_4$ is the part of $\alpha$ from the branch point to the intersection with $x_2$
		\item $x'_5$ is the part of $x_2$ from the intersection with $\alpha$ to its endpoint.
	\end{itemize}
	
	\begin{figure}[h!]
		\includegraphics[scale = 0.5]{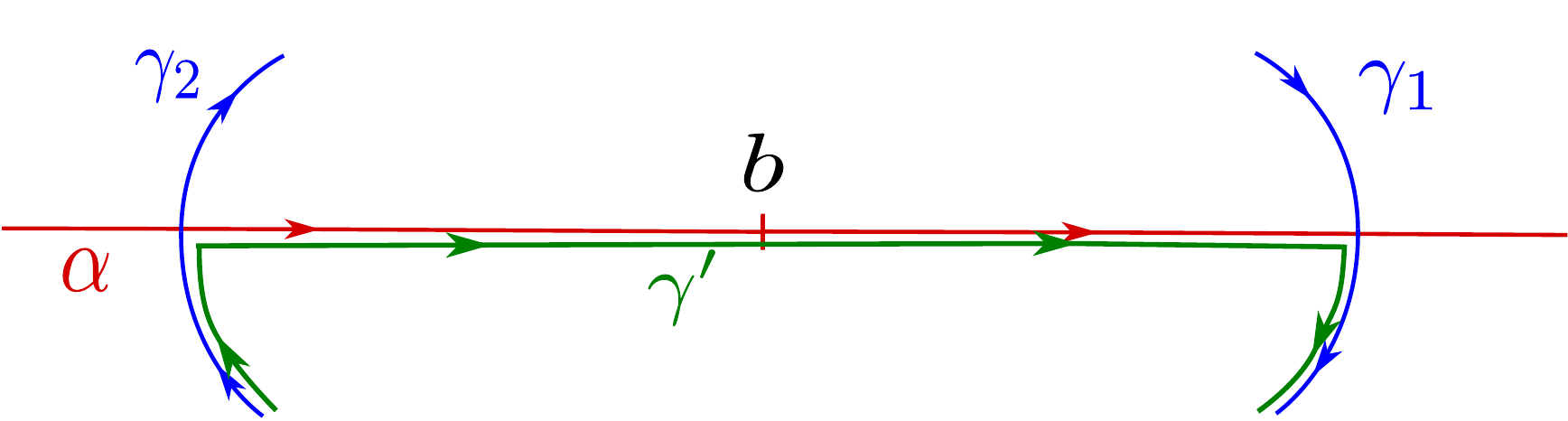}
		\caption{The path $x'$ added by the intersection with $\alpha$}
	\end{figure}

	Now let $\gamma : t\mapsto (x(t),v(t))$ be a path on $T'S$ such that the path $x$ on $S$ intersects only once the spectral network on a ray $\alpha\in\mathcal{W}$ at a time $t_0\in~ ]0,1[$. Let $\gamma_1 = (x_1,v_1)$ and $\gamma_2 = (x_2,v_2)$ be the standard lifts of $\gamma$ to $T'\Sigma$, with the same choice of numbering as above. Note that $\gamma_1$ and $\gamma_2$ do not intersect $T'\mathcal{W}$ in general, but $x_1$ and $x_2$ intersect $\alpha\in\mathcal{W}$. Let $x'$ be the path on $\Sigma$ obtained with the construction above. We now want a continuous map $v' : [0,1]\to \SS^1$ which coincide with the standard lifts $v_1$ and $v_2$ when $x'$ coincides with either $x_1$ or $x_2$. Without loss of generality, suppose $x$ is smooth at the intersection point with $\mathcal{W}$ and that the intersection is transverse. Then $x_1$ and $x_2$ are also smooth at their intersection points with $\alpha$. We say the intersection of $x_1$ with $\alpha$ is \defin{positively oriented} if $(\dot{x_1}(t_0),\dot{\alpha}(t_0))$ agrees with the orientation on $\Sigma$, \defin{negatively oriented} if not.
	\begin{remark}
		The positivity of the intersection of a path $(x,v)$ in $T'\Sigma$ with a ray of the spectral network is determined using the derivative of the underlying path $x$, and does not depend on the vector field $v$ on $x$.
	\end{remark}
	Let $v'$ be the concatenation of 5 paths $v'_1,\dots,v'_5$ defined as follows:
	\begin{itemize}
		\item $v'_1$ is the part of $v_1$ from its starting point to the intersection point with $\alpha$
		\item $v'_5$ is the part of $v_2$ from the intersection with $\alpha$ to its endpoint
		\item $v'_2$ is obtained by parallel transport with respect to the flat connection $\nabla$ on $\Sigma$ from the vector $v_1(t_0)$ along the path $x'_2$
		\item $v'_4$  is obtained by parallel transport with respect to $\nabla$ from the vector $v_2(t_0) = -v_1(t_0)$ along the path $\bar{x}'_4$
		\item $v'_3$ is the path $s^+_{v'_2(0)}$ in $T'_b\Sigma\simeq \SS^1$ if the intersection of $x_1$ with $\alpha$ is positively oriented, and $s^-_{v'_2(0)}$ if the intersection is negatively oriented.
	\end{itemize}

	\begin{figure}[h!]
		\includegraphics[scale = 0.5]{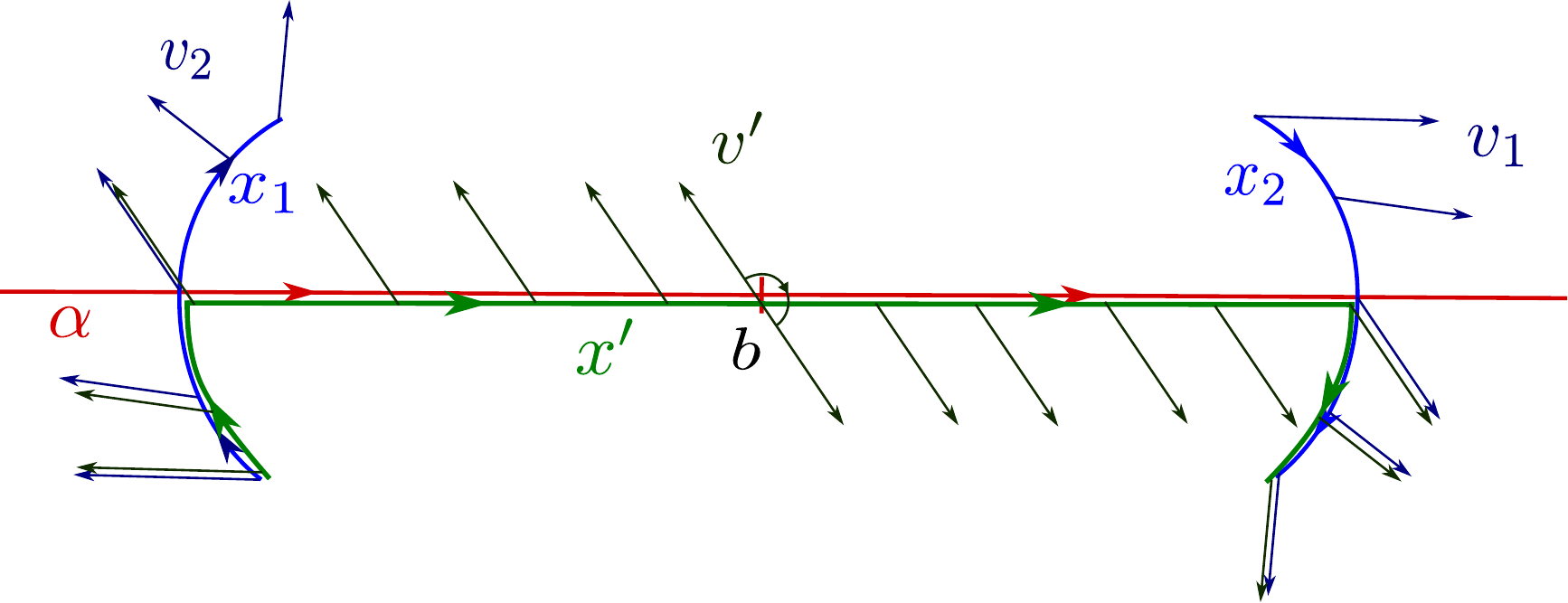}
		\caption{The path $(x',v')$ added by the intersection with $\alpha$. The intersection of $x_1$ with $\alpha$ is positively oriented if $\Sigma$ is oriented clockwise.}
	\end{figure}
	
	\begin{remark}\label{rk:homotopy_of_added_path}
		The resulting path $v'$ on $\SS^1$ is homotopic to $(-v^2_1).s^\pm_{v_1(t_0)}.v^1_1$ where $v_1^1 = v_1|_{[0,t_0]}$ and $v_1^2 = v_1|_{[t_0,1]}$. Note that for all path $w$ on $\SS^1$ from $\theta_0$ to $\theta_1$, we have
		\[ s_{\theta_1}^\pm.w \approx (-w).s_{\theta_0}^\pm \]
		so the path $v'$ is homotopic to $s^\pm_{v_1(1)}.v_1$.
	\end{remark}
	
	Let $\gamma' = (x',v')$ and $SN(\gamma)$ be the element $\gamma_1 + \gamma_2 + \gamma'\in \TPA(\Sigma)$. Let $\gamma$ be a path in $T' S$. We can write $\gamma$ as a concatenation of smaller paths $\gamma_1,\dots,\gamma_r$, each intersecting at most once the spectral network and for each of these small paths, apply the construction above to obtain $SN(\gamma_1),\dots,SN(\gamma_r)$ (if $\gamma_i$ does not intersect the spectral network, define $SN(\gamma_i)$ to be the sum of the two standard lifts of $\gamma_i$). Define the lift of $\gamma$ with respect to the spectral network $\mathcal{W}$ to be the product $SN(\gamma) = SN(\gamma_1)\dots SN(\gamma_r)\in \TPA(\Sigma)$.
	
	\begin{theorem}
		Let $\gamma_1$ and $\gamma_2$ be two homotopic paths in $T' S$. Then $SN(\gamma_1) = SN(\gamma_2)$.
		In particular, the map
		\[ SN : \begin{array}{rcl}
			\TPA(S)& \to &\TPA(\Sigma) \\
			\gamma&\mapsto & SN(\gamma)
		\end{array} \]
		is well-defined.
	\end{theorem}
	
	\begin{remark}
		The map $SN$ is not defined on the whole twisted path algebra of $S$ as paths with endpoints on a ray of the spectral network can not be lifted consistently, but we will never need to lift such paths. The subset of $\TPA(S)$ (resp. $\TPA(\Sigma)$) of elements where no term has an endpoint on $\mathcal{W}$ is a subring (see rem. \ref{rk:TPA_subalg}), and with a slight abuse of notation we will still denote it $\TPA(S)$ (resp. $\TPA(\Sigma)$).
	\end{remark}
	
	The theorem is a consequence of the two following lemmas:
	
	\begin{lemma}
		Let $\gamma = (x,v)$ be a path in $T' S$ that intersects exactly twice the same ray $\alpha$ of the spectral network and no other ray of $\mathcal{W}$, as in figure \ref{fig:SN5}. Then $SN(\gamma)=\gamma_1 + \gamma_2$ where $\gamma_1$ and $\gamma_2$ are the two standard lifts of $\gamma$.
	\end{lemma}
	
	\begin{figure}[h!]
		\includegraphics[scale = 0.6]{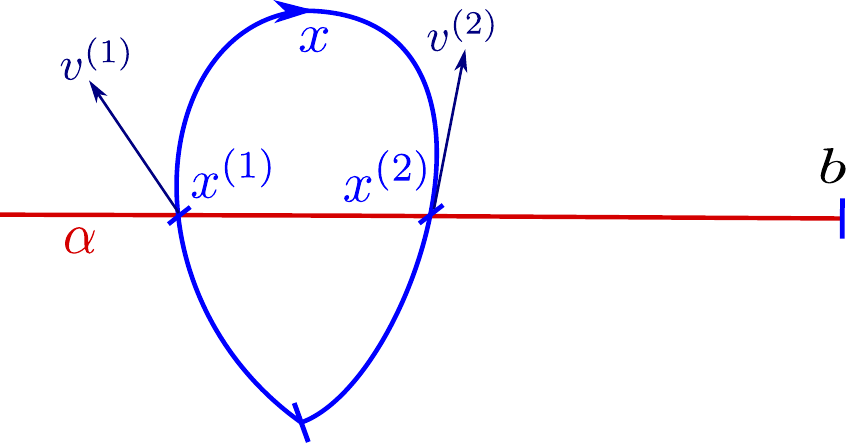}
		\caption{A loop intersecting twice the same ray of $\mathcal{W}$}
		\label{fig:SN5}
	\end{figure}
	
	\begin{proof}
		Let  $t_1<t_2$ be the two elements of the interval $\left[ 0,1\right] $ such that $x(t_1)$ and $x(t_2)$ are on $\alpha$. Let $(x^{(1)},v^{(1)}) = \gamma(t_1)$ and $(x^{(2)},v^{(2)})=\gamma(t_2)$, and let $\gamma_1=(x_1,v_1)$ and $\gamma_2=(x_2,v_2)$ the two standard lifts of $\gamma$, $\gamma_1$ being the lift intersecting $\alpha$ before the branch point. Then $SN(\gamma) = \gamma_1+\gamma_2+\gamma'+\gamma''$ where $\gamma' = (x',v')$ is such that $x'$ follow $\alpha$ from $x_1^{(1)}$ to $x_2^{(1)}$ and $\gamma''=(x'',v'')$ is such that $x''$ follows $\alpha$ from $x_1^{(2)}$ to $x_2^{(2)}$.
		
		\begin{figure}[h!]
			\includegraphics[scale = 0.6]{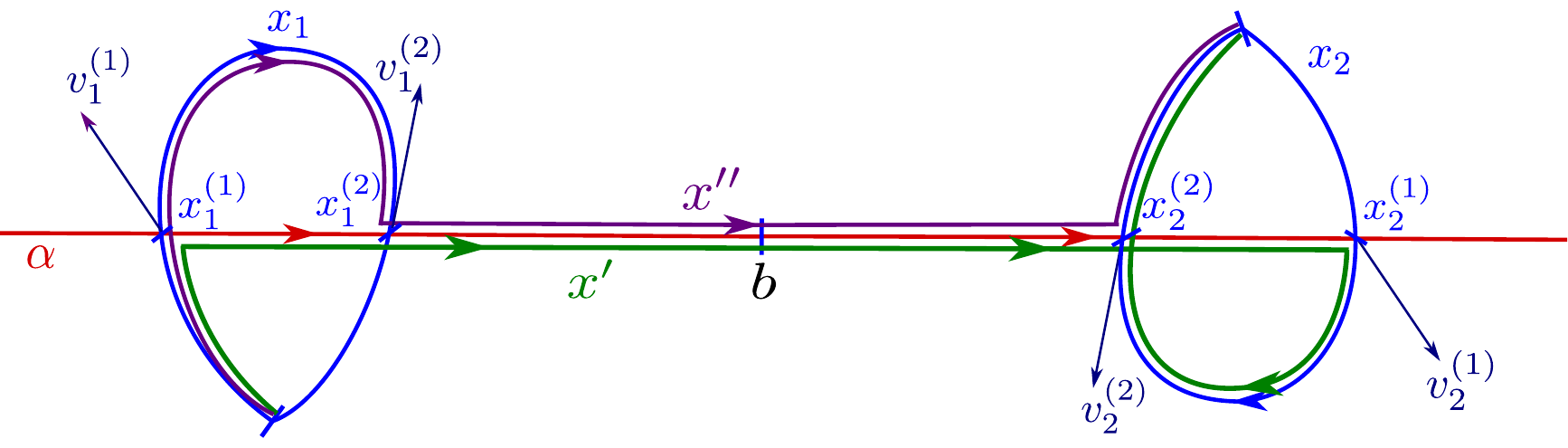}
			\caption{Spectral network lift of $\gamma$}
		\end{figure}
		
		In order to prove the lemma, we need to show that the two paths $\gamma'$ and $\gamma''$ added by the intersections with the spectral network cancel each other in $\TPA(\Sigma)$, i.e. that $\delta' +\delta'' =0$. For this, we need to show that $\bar\gamma''.\gamma'$ is homotopic to an odd power of $\delta_{x_1(0),v_1(0)}$. The paths $x'$ and $x''$ are homotopic on $\Sigma$ so the concatenation $\bar x''.x'$ is trivial. What is left is to show that $\bar v''.v'$ is homotopic to an odd power of $\delta^+$.
		
		Suppose the intersection of $x_1$ with $\alpha$ at $x_1^{(1)}$ is positive, the other case being symmetric. Then the intersection of $x_1$ with $\alpha$ at $x_1^{(2)}$ is negative. Then by remark \ref{rk:homotopy_of_added_path}, $v'\approx s^+.v_1$ and $v''\approx s^-.v_1$, so we have \begin{align*}
			\bar v''.v' & \approx\bar{v}_1.\overline{s^-}.s^+_{}.v_1 \\
			&\approx \bar{v}_1.\delta^+.v_1 \\
			&\approx \delta^+.\qedhere
		\end{align*}
	\end{proof}
	
	\begin{lemma}
		Let $m$ be a point in $S$ in a small neighborhood of a branch point $b$ but not on a ray of $\mathcal{W}$ and $\theta\in T'_m S$. Let $\gamma$ be a path homotopic to $e_{m,\theta}$ in $T' S$ that loops around a branch point $b$ in $S$, intersecting exactly once each of the three rays of $\mathcal{W}$ going out of $b$, as in figure \ref{fig:SN4}. Then $SN(\gamma)=e_{m_1,\theta_1} + e_{m_2,\theta_2}$ where $(m_1,\theta_1)$ and $(m_2,\theta_2)$ are the two lifts of $(m,\theta)$ in $T' \Sigma$.
	\end{lemma}
	
	\begin{figure}[h!]
		\includegraphics[scale = 0.9]{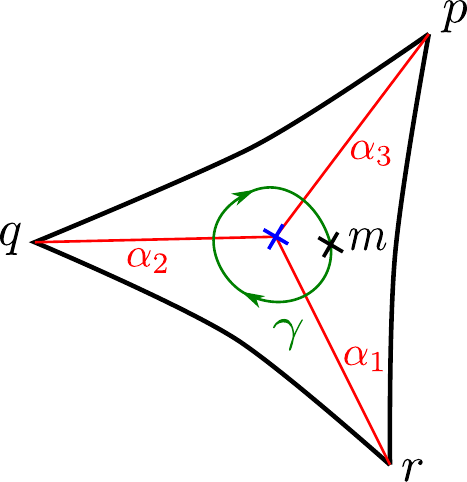}
		\caption{A small loop around a branch point.}
		\label{fig:SN4}
	\end{figure}
	
	\begin{proof}
		Suppose the path $\gamma$ is looping around $b$ in the direction given by the orientation of $\Sigma$, the other case being symmetric. Then all the intersections of the standard lifts of $\gamma$ with the spectral network in $\Sigma$ are positive.
		By applying the spectral network lifting rule to $\gamma$, we get 8 paths: the two standard lifts $\gamma_1 = (x_1,v_1)$ and $\gamma_2=(x_2,v_2)$, and 6 additional paths $\gamma'_1,\dots,\gamma'_6$ shown in figure \ref{fig:SN6}.
		\begin{figure}[h!]
			\includegraphics[scale=0.6]{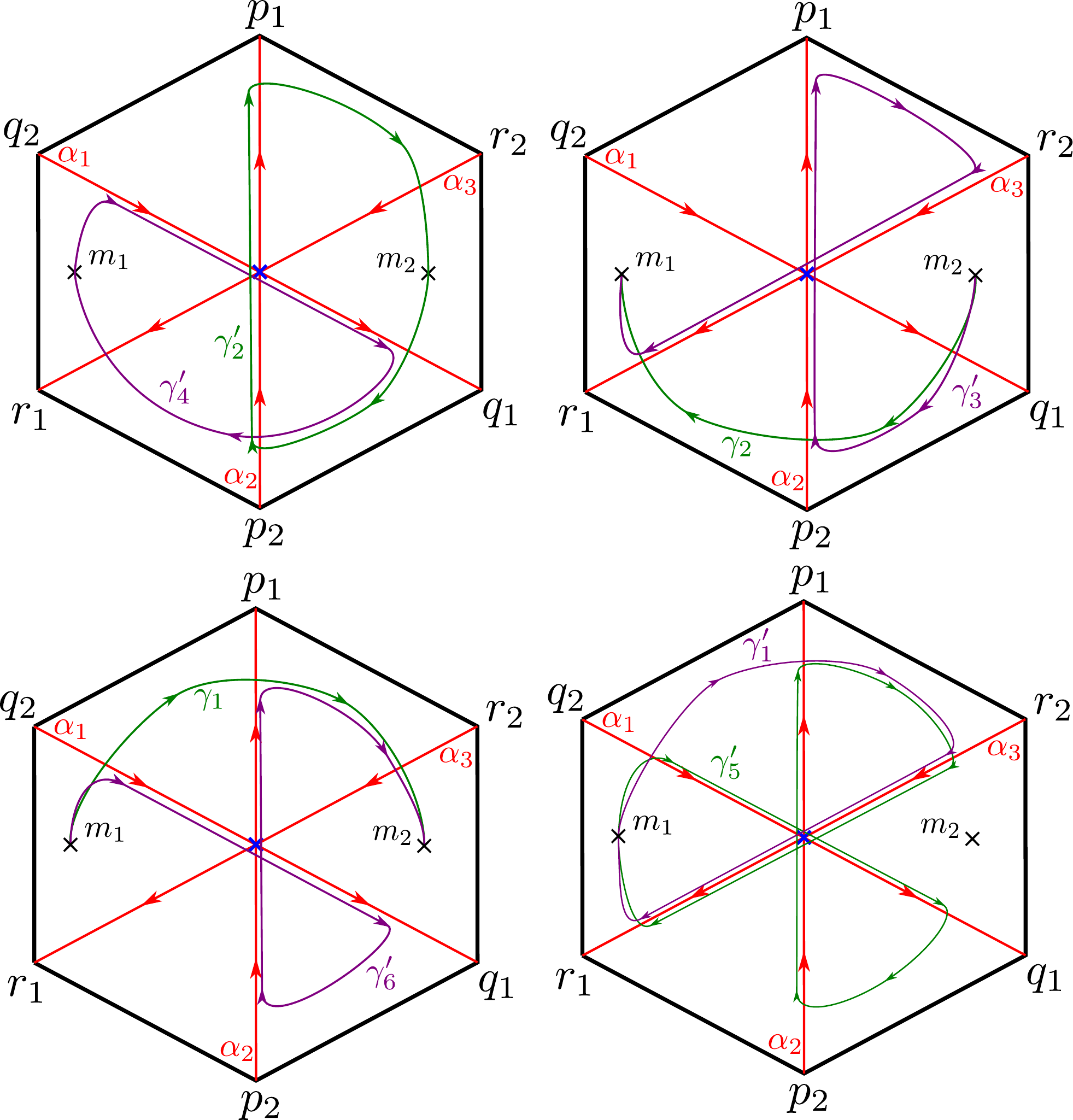}
			\caption{All 6 paths added by intersections with the spectral network, together with the standard lifts. On the upper left picture are the paths homotopic to trivial paths, and on the other are the remaining lifts, grouped as pairs of paths cancelling each other in $\TPA(\Sigma)$. Only the paths $x'_i$ on $\Sigma$ are drawn.}\label{fig:SN6}
		\end{figure}
		Let $\alpha_1,\alpha_2$ and $\alpha_3$ be the three ray of $\mathcal{W}$ intersected by $\gamma$, in that order. Let $\gamma_1$ be the standard lift of $\gamma$ intersecting $\alpha_1$ before the branch point, and let $(m_1,\theta_1)$ be its starting point and $(m_2,\theta_2)$ be its endpoint. We will label the spectral network lifts $\gamma'_i = (x'_i,v'_i)$ of $\gamma$ as follows:
		\begin{itemize}
			\item $\gamma'_1$ follows $\gamma_1$ until the intersection with $\alpha_3$, then $\alpha_3$, then $\gamma_2$ until its end
			\item $\gamma'_2$ follows $\gamma_2$ until the intersection with $\alpha_2$, then $\alpha_2$, then $\gamma_1$ until its end
			\item $\gamma'_3$ follows $\gamma_2$ until the intersection with $\alpha_2$, then $\alpha_2$, then $\gamma_1$ until the intersection with $\alpha_3$, then $\alpha_3$, then $\gamma_2$ until its end
			\item $\gamma'_4$ follows $\gamma_1$ until the intersection with $\alpha_1$, then $\alpha_1$, then $\gamma_2$ until its end
			\item $\gamma'_5$ follows $\gamma_1$ until the intersection with $\alpha_1$, then $\alpha_1$, then $\gamma_2$ until the intersection with $\alpha_2$, then $\alpha_2$, then $\gamma_1$ until the intersection with $\alpha_3$, then $\alpha_3$, then $\gamma_1$ until its end
			\item $\gamma'_6$ follows $\gamma_1$ until the intersection with $\alpha_1$, then $\alpha_1$, then $\gamma_2$ until the intersection with $\alpha_2$, then $\alpha_2$, then $\gamma_1$ until its end.
		\end{itemize}
		The paths $x'_1, x'_4$ and $x'_5$ are homotopic to the trivial path $e_{m_1}$, $x'_6$ is homotopic to $x_1$, $x'_2$ is homotopic to $e_{m_2}$ and $x'_3$ is homotopic to $x_2$. Since $\gamma$ is homotopic to $e_{m,\theta}$ and $x$ is looping around $b$ in the direction given by the orientation of $\Sigma$, we have $v_1 \approx s^-_{\theta_1}$ and $v_2 \approx s^-_{\theta_2}$. Using the same reasoning as above, we get the following:
		\begin{align*}
			v'_1 \approx s^+.v_1\approx e_{\theta_1} \\
			v'_2 \approx s^+.v_2\approx e_{\theta_2}\\
			v'_3 \approx s^+.s^+.v_2\approx \delta^+.v_2\\
			v'_4 \approx s^+.v_1\approx e_{\theta_1}\\
			v'_5 \approx s^+.s^+.s^+.v_1\approx \delta^+_{\theta_1}\\
			v'_3 \approx s^+.s^+.v_1\approx \delta^+.v_1
		\end{align*}
		So in $\TPA(\Sigma)$, we have:\begin{align*}
			\gamma_2 +\gamma'_3 = 0\\
			\gamma_1 + \gamma'_6 = 0\\
			\gamma'_1 +\gamma'_5 =0\\
			\gamma'_4 = e_{m_1,\theta_1}\\
			\gamma'_2 = e_{m_2,\theta_2}
		\end{align*}
		so
		\begin{proofeqn}
            SN(\gamma) = \gamma_1+\gamma_2+\gamma'_1+\gamma'_2+\gamma'_3+\gamma'_4+\gamma'_5+\gamma'_6= e_{m_2,\theta_2} + e_{m_2,\theta_2}.
        \end{proofeqn}
	\end{proof}
	
	\section{Partial abelianization of framed twisted local systems}\label{sec:ab_framed}
	
	In this section, we define framed and decorated twisted $\GL_2(A)$-local system over a punctured surface $S$ and describe a partial non-abelianization procedure for them. Using this, we describe the topology of the moduli space of framed and decorated twisted $\GL_2(A)$-local system that are transverse with respect to a fixed ideal triangulation of $S$.

	In this section, $A$ is a unital ring with a topology, and ring homomorphism are required to preserve unity elements. 
	
	Let $A^n$ be as the set of columns ($n\times 1$ matrices) endowed with the structure of a right $A$-module.

    \begin{df} 
    We make the following definitions:
        \begin{enumerate}
        \item An $n$-tuple $(x_1,\dots,x_n)$ for $x_1,\dots,x_n\in A^n$ is called \defin{basis} of $A^n$ if the map
        $$\begin{matrix}
        A^n & \to & A^n \\
        (a_1,\dots,a_n) & \mapsto & \sum_{i=1}^nx_ia_i
        \end{matrix}$$
        is an isomorphism of $A$-modules.
        \item The element $x\in A^n$ is called \defin{regular} if there exist $x_2,\dots,x_n\in A^n$ such that $(x,x_2,\dots,x_n)$ is a basis of $A^n$.
        \item $\ell\subseteq A^n$ is called an \defin{$A$-line} if $\ell=xA$ for a regular $x\in A^n$. We denote the space of $A$-lines of $A^n$ by $\PP(A^n)$.
        \item Regular elements $x_1,\dots,x_k\in A^n$ for $k\leq n$ are called \defin{linearly independent} if there exist $x_{k+1},\dots ,x_n\in A^n$ such that $(x_1,\dots,x_n)$ is a basis of $A^n$.
        \item Two $A$-lines $\ell,m$ are called \defin{transverse} if $\ell=xA$, $m=yA$ for linearly independent $x,y\in A^n$.
        \end{enumerate}
    \end{df}
	
	Let $M_n(A)$ be the ring of all $n\times n$-matrices with entries in $A$, and $\GL_n(A)$ be the group of all invertible matrices of $M_n(A)$. Then $\GL_n(A)$ acts on $A^n$ by the left multiplication. 
	
	\begin{definition}
		A \defin{$\GL_n(A)$-local system} over a smooth manifold $X$ is a $A^n$-bundle over $X$ equipped with a flat connection.
	\end{definition}
	
	\begin{df}
	    Let $U$ be an open subset of $X$. A \defin{regular $A$-subbundle} $L$ of a $\GL_n(A)$-local system $\mathcal L$ over $U$ is a subbundle of $\mathcal L$ such that for every $p\in U$ there exists a neighborhood $U_p$ containing $p$ and a local trivialization
	    $$\Phi_p\colon \mathcal L|_{U_p}\to U_p\times A^n$$
	    such that $\Phi_p(L|_{U_p})=U_p\times\ell$ where $\ell$ is an $A$-line in $A^n$. 
	    
	    A section $v\colon U\to \mathcal L$ is \defin{regular} if $vA$ is a regular $A$-subbundle of $\mathcal L|_U$. 
	\end{df}
	
	\subsection{Twisted local systems}
	
	In this section, $X$ denotes either $S$ or $\Sigma$.
	
	\begin{definition}
		A \defin{twisted $\GL_n(A)$-local system} on $X$ is a flat $A^n$-bundle over $T'X$ with monodromies around the fibers of the natural projection $T'X\to X$ equal to $-\Id$.
	\end{definition}
	
	Let $x_0\in X$ and $v\in T'_{x_0}X$. The natural projection $T'X\to X$ has fiber homeomorphic to $\SS^1$, and we have the short exact sequence
	$$ 1\to \pi_1(T_{x_0}'X,(x_0,v)) \to \pi_1(T'X,(x_0,v)) \to \pi_1(X,x_0) \to 1$$
	with $\pi_1(T_{x_0}'X,(x_0,v))$ being isomorphic to $\ZZ$, generated by the loop $\delta_{x_0,v}^+$ going around the fiber over $x_0$ once in the direction given by the orientation of $X$. This extension is central because $X$ is oriented.
	
	Since $X$ is not closed, the group $\pi_1(X)$ is free, so the sequence above splits. The choice of a splitting corresponds to the choice of a non-vanishing vector field on $X$. Let $\pi^s_1(X)$ denote the quotient of $\pi_1(T'X,(x_0,v))$ by the normal subgroup $2\ZZ\subset\ZZ\simeq \pi_1(T_{x_0}'X,(x_0,v))$, so we have the short exact sequence
	$$ 1\to \ZZ/2\ZZ \to \pi^s_1(X) \to \pi_1(X,x_0) \to 1$$
	that once again splits. Note that this second sequence also splits when $X$ is closed (for instance for $X=\bar\Sigma$ and $\bar{\Sigma}$ has no boundary) since a closed surface of negative Euler characteristic always admits a vector field with zeroes of even indices only.
	
	\begin{proposition}\label{prop:tw_hilb_corr}
		The set of twisted $\GL_n(A)$-local systems on $X$ up to isomorphism is in 1:1 correspondence with the set of representations $\rho: \pi^s_1(X)\to \GL_n(A)$ such that $\rho(\delta_{x_0,v}^+)=-\Id$, up to the action of $\GL_n(A)$ by conjugation.
	\end{proposition}
	
	If $\mathcal{L}$ is a twisted local system on $X$ and $\gamma$ is a path on $T'X$, the flat connection defines a holonomy map $m_\gamma$ from $\mathcal{L}_{\gamma(0)}$ to $\mathcal{L}_{\gamma(1)}$. Moreover, the path $\delta_{x,\theta}$ induces the linear map $-\Id$ on $\mathcal{L}_{x,\theta}$ by definition of a twisted local system. Thus, if $\gamma = \gamma_1+\dots+\gamma_r + \mathcal{I}\in \TPA(X)$ where all the $\gamma_i\in \Path(T'X)$ have the same extremities, the holonomy map $m_\gamma = m_{\gamma_1}+\dots+m_{\gamma_r}: \mathcal{L}_{\gamma(0)}\to \mathcal{L}_{\gamma(1)}$ is well-defined
	(if there is more than one term in $\gamma$ the holonomy map $m_\gamma$ may not be an isomorphism).
	However, if $\gamma_1$ and $\gamma_2$ do not have the same extremities, it is not possible to associate an element of $M_n(A)$ to $\gamma_1+\gamma_2$, which is a problem we need to solve in order to consider representations of $\TPA(X)$. To make a link between twisted local systems and representations of $\TPA(S)$, we first need to modify the ring $M_n(A)$ to solve this issue of endpoints. Since multiplication in $\TPA(X)$ is zero for paths whose extremities do not match, we need a ring with the same behavior.
	
	\begin{definition}
		Let $A$ be a unital ring and $E\subset X$ any non-empty subset. Let $A_E$ be the ring $A^{(E\times E)}$ of finite formal sums of elements of the form $ a_{(p,q)},a\in A, p,q\in E$, endowed with the multiplication defined as follows:
		\begin{itemize}
			\item $\forall a,b\in A,\forall x,y,z\in E, a_{(x,y)}.b_{(y,z)} = (a.b)_{(x,z)}$
			\item $\forall a,b\in A,\forall x,y,z,t\in E,y\neq z, a_{(x,y)}.b_{(z,t)} = 0$
		\end{itemize}
	\end{definition}
	
	The elements of this ring are copies of elements of $A$ indexed by pairs of points in $E$, thought as ``endpoints'' of these elements. The sum of two elements is a formal sum except when the indices match, it then agrees with the sum in $A$. The multiplication of two elements is made, so it agrees with the composition of paths: multiplication of two elements with ``non-composable'' indices is zero and multiplication with ``composable'' indices agrees with the one on $A$, and the index of the result is the composition of the indices.
	
	The ring $A_E$ contains many isomorphic copies of $A$ as subrings: for all $x\in E$,
	\[ A_x := \left\lbrace a_{(x,x)}~|~a\in A\right\rbrace  \]
	is a subring of $A_E$ isomorphic to $A$. Note however that the ring $A_E$ is not unital if $E$ is infinite, but contains many idempotent elements.
	
	Let $\TPA_\mathcal{T}(S) = \TPA_{I_{\mathcal{T}}(S)}(S)$ be the subring of $\TPA(S)$ of paths with endpoints in $I_{\mathcal{T}}(S)$ described in Remark~\ref{rk:TPA_subalg}. Similarly, let $\TPA_\mathcal{T^*}(\Sigma) = \TPA_{I_{\mathcal{T^*}}(\Sigma)}(\Sigma)$. For any unital ring, let $A_\mathcal{T} = A_{I_{\mathcal{T}}(S)}$ and $A_\mathcal{T^*} = A_{I_{\mathcal{T^*}}(\Sigma)}$. In the following, $\TPA_\mathcal{T^{(*)}}(X)$ denote either $\TPA_\mathcal{T}(S)$ or $\TPA_\mathcal{T^*}(\Sigma)$ and similarly $A_\mathcal{T^{(*)}}$ denote either $A_\mathcal{T}$ or $A_\mathcal{T^*}$. Since $I_{\mathcal{T}}(S)$ and $I_{\mathcal{T^*}}(\Sigma)$ are finite, $\TPA_\mathcal{T^{(*)}}(X)$ and $A_\mathcal{T^{(*)}}$ are unital, the units elements being respectively $\sum_{x\in I_{\mathcal{T^{(*)}}}(X)} e_x$ and $\sum_{x\in I_{\mathcal{T^{(*)}}}(X)} 1_{(x,x)}$. There is then a diagonal embedding
	$$\begin{array}{rcl}
	A^{\# I_{\mathcal{T^{(*)}}}(X)}&\to& A_\mathcal{T^{(*)}}\\
	(a_x)_{x\in I_{\mathcal{T^{(*)}}}(X)}&\mapsto&\sum_{x\in I_{\mathcal{T^{(*)}}}(X)} (a_x)_{(x,x)}
	\end{array}.$$
	For every $x\in I_{\mathcal{T^{(*)}}}(X)$, there is an injective group homomorphism
	\[ \pi^s_1(X,x) \to \TPA_x(X)^{\times}\subset \TPA_\mathcal{T^{(*)}}(X). \]
	Two elements $a,b\in A_\mathcal{T^{(*)}}$ are said \defin{conjugated} if there exists an invertible element $u$ in $A^{\# I_{\mathcal{T^{(*)}}}}$ such that
	$ b  = u.a.u^{-1}. $
	This is an equivalence relation.
	
	\begin{proposition}\label{prop:tw_loc_repr}
		Let $A$ be a unital ring. There is a 1:1 correspondence between the set of twisted $\GL_n(A)$-local systems on $X$ up to isomorphism and the set of ring homomorphisms $\TPA_{\mathcal{T^{(*)}}}(X)\to M_n(A)_\mathcal{T^{(*)}}$ up to the action of $\GL_n(A)^{\# I_{\mathcal{T^{(*)}}}(X)}$ by conjugation.
	\end{proposition}
	
	\begin{proof}
		Given a twisted $\GL_n(A)$-local system $\mathcal{L}$ on $X$, for all $x\in I_{\mathcal{T^{(*)}}}(X)$ choose a basis of the fiber of $\mathcal{L}$ over $x$. The map
		$$\varphi : \begin{array}{rcl}
			\TPA_\mathcal{T^{(*)}}(X)&\to& M_n(A)_\mathcal{T^{(*)}}\\
			\sum \gamma &\mapsto &\sum (Hol_\mathcal{L}(\gamma))_{(t(\gamma),s(\gamma))}
		\end{array}$$
		is a ring homomorphism, where $s(\gamma)$ (resp. $t(\gamma)$) is the source (resp. the end) of $\gamma$ (which are in $I_{\mathcal{T^{(*)}}}(X)$), and $Hol_\mathcal{L}(\gamma)$ is the holonomy of $\gamma$ in $\mathcal{L}$ in the corresponding bases. The conjugacy class of $\varphi$ does not depend on the choices of the bases.
		
		Conversely, let $\varphi$ be the conjugacy class of a representation $\TPA_\mathcal{T^{(*)}}(X)\to M_n(A)_\mathcal{T^{(*)}}$, and let $x\in I_{\mathcal{T^{(*)}}}(X)$. Then $\TPA_x(X)$ contains an isomorphic copy of $\pi^s_1(X,x)$ and the restriction of $\varphi$ to $\pi^s_1(X,x)$ yield a representation $\pi^s_1(X,x)\to \GL_n(A)$ mapping $\delta_{x}^{\pm}$ to $-\Id$, which define a unique isomorphism class of twisted $\GL_n(A)$-local system by Proposition~\ref{prop:tw_hilb_corr}, having holonomies described by $\varphi$.
	\end{proof}
	
	\subsection{Framing and decoration}
	
	Let $\mathcal{L}$ be a twisted $\GL_2(A)$-local system on $S$. We say that $\mathcal{L}$ is \defin{peripherally parabolic} if for every puncture $p\in P$ there exist a parallel regular $A$-subbundle of $\mathcal{L}$ over $T'\beta_p$. A choice of such a parallel regular subbundle $L_p\subset \mathcal{L}_p\to T'\beta_p$, where $\mathcal{L}_p:=\mathcal{L}|_{T'\beta_p}$ for every $p\in P$ is called a \defin{framing} of $\mathcal{L}$. Since $L_p$ is parallel, on the quotient bundle $\mathcal{L}_p/L_p$ (which is an $A$-bundle) over $T'\beta_p$ the flat connection is also well-defined. A \defin{framed twisted $\GL_2(A)$-local system} is a pair $(\mathcal{L},(L_p)_{p\in P})$ where $(L_p)_{p\in P}$ is a framing of $\mathcal L$.
	
	Let $\mathcal{T}$ an ideal triangulation of $S$. We say a framed twisted local system $(\mathcal{L},(L_p)_{p\in P})$ is \defin{$\mathcal{T}$-transverse} if for every edge of the triangulation two subbundles corresponding to two ends of the edge are transverse.
	
	A twisted $\GL_2(A)$-local system $\mathcal{L}$ over a decorated surface $(S,\mathcal D)$ is called \defin{peripherally unipotent} if for every $\beta_p\in \mathcal D$, $p\in P$ there exist a parallel regular section $v_p$ of $\mathcal{L}$ along $T'\beta_p$ and a parallel regular section $w_p$ of the bundle $\mathcal{L}_p/L_p$ along $T'\beta_p$, where $L_p$ is the $A$-subbundle of $\mathcal{L}_p$ spanned by $v_p$. If for every $\beta_p$ such parallel regular sections $v_p$ of $\mathcal{L}_p$ and $w_p$ of $\mathcal{L}_p/L_p$ along $T'\beta_p$ are chosen, then $(\mathcal{L}, (v_p)_{p\in P}, (w_p)_{p\in P})$ is called a \defin{decorated} twisted $\GL_2(A)$-local system.
	
	\subsection{Non-abelianization of twisted local systems}
	
	In Section~\ref{sec:tw_ab}, we constructed an algebra homomorphism $SN : \TPA(S)\to \TPA(\Sigma)$. This homomorphism restricts to a ring homomorphism
	\[ SN : \TPA_\mathcal{T}(S)\to \TPA_\mathcal{T^*}(\Sigma) \]
	as mentioned in Remark~\ref{rk:TPA_subalg}. Let $\gamma\in \TPA_\mathcal{T}(S)$ be a path from $p$ to $q$, $p,q\in I_\mathcal{T}(S)$, and let $p_1,p_2$ be the two lifts of $p$ to $\Sigma$, and $q_1,q_2$ the two lifts of $q$, with $p_1,q_1$ being the sinks and $p_2,q_2$ being the sources. Then \[ SN(\gamma) = \gamma_{1,1}+\gamma_{1,2}+\gamma_{2,1}+\gamma_{2,2} \]
	where $\gamma_{j,i}$ is the sum of all terms of $SN(\gamma)$ from $p_i$ to $q_j$ ($\gamma_{i,j}$ may be 0). Instead of a formal sum, it will be more convenient to see $SN(\gamma)$ as a 2 by 2 matrix with coefficients in $\TPA_\mathcal{T^*}( \Sigma)$. The definition of the multiplication on $\TPA_\mathcal{T^*}(\Sigma)$ makes it so the map:
	\[ SN : \begin{array}{rcl}
		\TPA_\mathcal{T}(S)&\to& M_2(\TPA_\mathcal{T^*}(\Sigma))\\[1.2ex]
		\gamma&\mapsto&\begin{pmatrix}
			\gamma_{1,1}&\gamma_{1,2}\\
			\gamma_{2,1}&\gamma_{2,2}
		\end{pmatrix}
	\end{array} \]
	is a ring homomorphism.
	We also have a ring homomorphism $$\pi_A : \begin{array}{rcl}
		M_2(A_{\mathcal{T^*}})&\to& M_2(A)_{\mathcal{T}}\\[1.2ex]
		\begin{pmatrix}
			a_{q_1,p_1}&b_{q_1,p_2}\\
			c_{q_2,p_1}&d_{q_2,p_2}
		\end{pmatrix}&\mapsto&\begin{pmatrix}
			a&b\\
			c&d
		\end{pmatrix}_{q,p}
	\end{array}.$$
	Note that we can always write an element of $M_2(A_{\mathcal{T^*}})$ as the sum of elements of the form $$\begin{pmatrix}
			a_{q_1,p_1}&b_{q_1,p_2}\\
			c_{q_2,p_1}&d_{q_2,p_2}
		\end{pmatrix}$$
	(possibly with some coefficients equal to 0).
	
	\begin{prop}
		Let $\mathcal{E}$ be a twisted $A^\times$-local system over $\Sigma$ and let $\varphi : \TPA_\mathcal{T^*}(\Sigma)\to A_\mathcal{T^*}$ the corresponding ring homomorphism given by Proposition~\ref{prop:tw_loc_repr}. Then the ring homomorphism
		\[ \psi = \pi_A\circ M_2(\varphi)\circ SN : \TPA_\mathcal{T}(S)\to M_2(A)_\mathcal{T} \]
		corresponds to a peripherally parabolic twisted $\GL_2(A)$-local system on $S$, together with a $\mathcal{T}$-transverse framing.
	\end{prop}
	
	\begin{proof}
		Let $\mathcal{E}$ be a twisted $A^\times$-local system on $\Sigma$, and let $\mathcal{L}$ be the $\GL_2(A)$-local system obtained on $S$.
		We need to show that $\mathcal{L}$ admits a flat section on any peripheral curve $\beta_p$ on $S$, i.e. that the monodromy along $\beta_p$ is upper triangular in some basis. Let $p\in P$ and $p_1,p_2$ the lifts of $p$ to $\Sigma$, $p_1$ being the sink and $p_2$ the source. Let $q\in I_{\mathcal{T}}(S)\cap\beta_p$ and $q_1,q_2$ the lifts of $q$ to $\Sigma$, $q_i\in\beta_{p_i}$. We will assume  $\beta_p$ is a loop based on $q$. The fiber $\mathcal{L}_{q}$ of $\mathcal{L}$ over $q$ can be identified with the direct sum $\mathcal{E}_{q_1}\oplus\mathcal{E}_{q_2}$ of the fibers of $\mathcal{E}$ over $q_1$ and $q_2$. Every ray of the spectral network $\mathcal{W}$ crossed by $\beta_p$ on $S$ lifts to a ray from $p_2$ to $p_1$ on $\Sigma$. This means that the lifts added by the spectral network all go from $q_1$ to $q_2$, so the image of $\beta_p$ via $SN : \TPA_\mathcal{T}(S)\to M_2(\TPA_\mathcal{T^*}(\Sigma))$ is upper triangular. Then $\psi(\beta_p)\in P_\mathcal{T}(M_2(A))$, the monodromy of $\beta_p$, is also upper triangular. The line $\mathcal{E}_{q_1}\subset \mathcal{L}_q$ is preserved by the peripheral monodromy which means that the parallel transport of $\mathcal{E}_{q_1}$ along $\beta_p$ defines a framing $L_p\subset\mathcal{L}_{\beta_p}$ around $p$. This framing is $\mathcal{T}$-transverse because for every edge $\gamma$ of $\mathcal{T}$ from $p$ to $q$, $p,q\in P$, the map $L_p \to \mathcal{L}_q/L_q$ is the holonomy of $\mathcal{E}$ along one of the lifts of $\tau_\gamma$ so it is an isomorphism.
	\end{proof}
	
	The twisted $\GL_2(A)$-local system $\mathcal{L}$ on $S$ obtained from a twisted $A^\times$-local system $\mathcal{E}$ on $\Sigma$ via this construction is called the \defin{non-abelianization} of $\mathcal{E}$. In the next part, we construct an inverse construction.
	
	%
	
	\subsection{Partial abelianization of transverse framed local systems}\label{sec:framed}
	
	Let $(\mathcal{L},(L_p)_{p\in P})$ be a $\mathcal{T}$-transverse framed twisted $\GL_2(A)$-local system on $S$.
	Our goal is to construct a twisted $A^\times$-local system $\mathcal{E}\to T'\Sigma$
	that provides under the non-abelianization procedure the initial local system $\mathcal{L}$.

    Let $\pi'\colon\Sigma'\to S'$ be the covering as in Section~\ref{sec:spectral_networks}.	A transverse framed twisted $\GL_2(A)$-local system $\mathcal{L}$ over $S$ gives rise to a transverse framed twisted $\pi_1(S)$-equivariant $\GL_2(A)$-local system $\mathcal{L}'$ over $S'$. We obtain a parallel $A$-subbundle $L'_p$ over the preimage $T'U_p$ under $T'S'\to S'$ of every standard neighborhood $U_p$ of every puncture $p$ of $S'$. And similarly to the discussion above, on the quotient $A$-bundle $\mathcal{L}'_p/L'_p$ over $T'U_p$ the flat connection is well-defined.
	
	The spectral network $\mathcal W$ on $\Sigma$ (resp. $\mathcal W'$ on $\Sigma'$) divides the set of punctures of $\Sigma$ (resp. $\Sigma'$) in two classes: sinks of $\mathcal W$ (resp. $\mathcal W'$) and sources of $\mathcal W$ (resp. $\mathcal W'$). For every sink $p$ of $\Sigma'$ we define a flat $A$-bundle over $T'U_p$ as the pull-back of the $A$-subbundle $L'_{\pi'(T'U_p)}$. For every source $p$ of $\Sigma'$ we define a flat $A$-bundle over $T'U_p$ as the pull-back of the $A$-bundle $\mathcal{L}'/L'_{\pi'(T'U_p)}$.
	
	To construct a twisted flat $A$-bundle over $\Sigma'$ we need to ``glue'' the standard neighborhoods along cells of the (lifted) spectral network. For this, we notice that two standard neighborhoods that share a cell always correspond to punctures of different classes. Along the interior of such a cell $c$ two $A$-bundles are defined: $L'_{\pi'(T'U_p)}$ corresponding to a sink $p$ of $\Sigma'$ and $\mathcal{L}'/L'_{\pi'(T'U_q)}$ corresponding to a source $q$ of $\Sigma'$.
	
	By transversality, for every point $z\in T'c\subset T'\Sigma'$ the $A$-submodules $L'_{\pi'(T'U_p)}(\pi'(z))$ and $L'_{\pi'(T'U_q)}(\pi'(z))$ of $\mathcal{L}$ are transverse. That means that for all $z\in T'c$ the natural projection map $$a_{qp}(\pi'(z))\colon L'_{\pi'(T'U_p)}(\pi'(z))\to \mathcal{L}'/L'_{\pi'(T'U_q)}(\pi'(z))$$
	is an isomorphism. So we can identify $v\in L_{\pi'(T'U_p)}(\pi'(z))$ with $a_{qp}(\pi'(z))(v)\in \mathcal{L}/L_{\pi'(T'U_q)}(\pi'(z))$ for all $z\in T'c$. Since $a_{qp}(\pi'(z))$ is constant in any parallel frames of $L'_{\pi'(T'U_p)}$ and $\mathcal{L}'/L'_{\pi'(T'U_q)}$ along $T'c$, this provides a constant change of local trivialization over $T'c$. In other words, we obtain a flat $A$-bundle over $\Sigma'\bs B'$ that we denote by $\mathcal{E}'\to \Sigma'\bs B'$. The construction of $\mathcal{E}$ is obviously $\pi_1(S)$ equivariant and, therefore, defines a flat $A$-bundle $\mathcal{E}\to T'(\Sigma\bs B)$.
	
	The Lemma~\ref{lem:triang_rel} below provide that the holonomies around branch points are trivial. This implies that the bundle $\mathcal{E}$ can be extended also over $B$, i.e. we obtain a bundle $\mathcal{E}\to T'\Sigma$. The twisted $A^\times$-local system obtained on $\Sigma$ is called the \defin{abelianization} of $\mathcal{L}$.
	
	Those processes are inverse to each other by construction.

	We have thus shown:
	
	\begin{theorem}
		The abelianization and non-abelianization processes define a bijection between the set of twisted $A^\times$-local systems on $\Sigma$ up to isomorphism and the set of framed $\mathcal{T}$-transverse twisted local systems on $S$ up to isomorphism.
	\end{theorem}
	
	\begin{remark}
		The construction of $\Sigma$ depends on $\mathcal{T}$, which is implicit in the theorem.
	\end{remark}
	
	\section{Partial abelianization of decorated twisted local systems}\label{sec:ab_decorated}
	
	In this section, we apply the above construction to decorated twisted $\GL_2(A)$-local systems. This construction gives a geometrical representation of the non-commutative algebra introduced in \cite{BR}, as well as geometrical proof of the non-commutative Laurent phenomenon. We also use this construction to describe the topology of the space of framed or decorated twisted $\GL_2(A)$-local systems.
    
	\subsection{Kashiwara-Maslov map}\label{sec:Kashiwara}
	
	Let $\ell_1$, $\ell_2$, $\ell_3$ be pairwise transverse $A$-lines in $A^2$. We denote $a_{ji}\colon \ell_i\to A^2/\ell_j$ the projection maps for all $i,j\in\{1,2,3\}$, $i\neq j$. By transversality, $a_{ij}$ are $A$-linear isomorphisms.
	
	The following lemma is immediate:
	\begin{lem}\label{lem:triang_rel}
		The map $-a_{31}^{-1}a_{32}a_{12}^{-1}a_{13}a_{23}^{-1}a_{21}\colon \ell_1\to \ell_1$ is the identity map. 
	\end{lem}
	
	The map $\mu_1^{23}:=a_{13}a_{23}^{-1}a_{21}\colon \ell_1\to A^2/\ell_1$ is called the \defin{Kashiwara-Maslov map} of the triple of $A$-lines $(\ell_1,\ell_2,\ell_3)$. Notice that $\mu_1^{23}=-\mu_1^{32}$.
	
	Let $\pi\colon \Sigma\to S$ be a ramified covering associated to an ideal triangulation $\mathcal T$ as before. Let $\mathcal L\to T'S$ be a framed $\mathcal T$-transverse twisted local system over $S$ and $\mathcal E\to T'\Sigma$ be the twisted $A^\times$-local system obtained from $\mathcal L$ by the partial abelianization described in Section~\ref{sec:framed}. Let $\tau\subset S$ be a triangle of $\mathcal T$ that is incident to punctures $p_1,p_2,p_3$, and the orientation of the triangle agrees with the orientation of the triple $(p_1,p_2,p_3)$. Let $H=\pi^{-1}(\tau)$ be the hexagon of $\Sigma$ that covers $\tau$. The hexagon $H$ is divided in six (open) triangles by lines of the spectral network, as on Figure~\ref{SN-triangle}. As before, let $L_i$ be a parallel $A$-subbundle of $\mathcal L\to T'\tau$ corresponding to the puncture $p_i$, $i\in\{1,2,3\}$. By construction of the local system on $\Sigma$, over every triangle of $H$ two $A$-bundles $L_i$ and $\mathcal L/L_j$, $i,j\in\{1,2,3\}$ are defined. These two bundles are glued along this triangle by the projection map $a_{ji}\colon L_i\to A^2/L_j$. We denote this triangle by $t_{ji}$. Let now $p\in T't_{21}$, then $\theta(p)\in T't_{12}$, where $\theta\colon T'\Sigma\to T'\Sigma$ is the involution associated with the covering $\pi\colon\Sigma\to S$.
	
	We take a path $\gamma\colon [0,1]\to T'H$ such that $\gamma(0)=p$, $\gamma(1)=\theta(p)$ and such that $(\theta\circ\gamma).\gamma$ is a loop in $T'H$ homotopic to the loop going once in the positive direction around the fiber of $T'\Sigma\to \Sigma$.
	
	The following proposition follows directly from the construction of the $A^\times$-local system over $T'\Sigma$:
	\begin{prop}
		The parallel transport along the path $\gamma$ agrees with the Kashiwara-Maslov map $\mu_1^{23}\colon L_1(\pi(p))\to \mathcal L/L_1(\pi(p))$.
	\end{prop}
	
	\subsection{Partial abelianization of transverse decorated local systems}
	
	Applying the construction of Section~\ref{sec:framed} to a peripherally unipotent local system, we get an $A^\times$-local system $\mathcal{E}\to T'\Sigma$. Moreover, the existence of sections $v_p$ and $w_p$ over $T'\beta_p$ for all $p\in P$ induces that the holonomies of $\mathcal{E}$ around punctures of $\Sigma$ are all trivial, i.e. the bundle $\mathcal{E}$ is trivial over the punctured disk around $p$ bordered by $\beta_p$. That means that the local system $\mathcal{E}\to T'\Sigma$ can be uniquely extended to the local system over $T'\bar\Sigma$. For simplicity, slightly abusing the notation, we will write $\mathcal{E} \to T'\bar\Sigma$.
	
	A decoration of $\mathcal{L}$ provides additionally a parallel section of $\mathcal{E}_p\to T'\bar\Sigma_p$, where $\bar\Sigma_p=\Sigma_p\cup\{p\}$ for all $p\in\pi^{-1}(P)$. We call the set of all those parallel sections a \defin{decoration} of the twisted $A^\times$-local system $\mathcal{E}$.
	
	\begin{theorem}
		The abelianization and non-abelianization processes define a bijection between the set of decorated twisted $A^\times$-local systems on $\Sigma$ with trivial monodromy around punctures up to isomorphism and the set of decorated $\mathcal{T}$-transverse twisted local systems on $S$ up to isomorphism.
	\end{theorem}

	\subsection{Non-commutative \texorpdfstring{$\mathcal{A}$}{A}-coordinates and partial abelianization}\label{sec:A-coord}
	
	Let $\mathcal{D}$ be a decoration of $S$ and let $\mathcal{T}$ be a triangulation of $S$. Let $(\mathcal{L}, (v_p)_{p\in P}, (w_p)_{p\in P})$ be a decorated twisted $\GL_2(A)$-local system on the surface $S$, and assume $\mathcal{L}$ is $\mathcal{T}$-transverse. Then for every arc $\gamma$ of $\mathcal{T}$ from $p\in P$ to $q\in P$, we can trivialize the $\GL_2(A)$-local system $\mathcal{L}$ over $T'\gamma$ and the $A$-subbundles spanned by the flat sections $v_p$ and $v_q$ are transverse. The natural projection
	\[ a_{\gamma} : L_p = \Span(v_p) \to \mathcal{L}/L_q = \Span(w_q) \]
	is an isomorphism, and we can identify it with its (1 by 1) matrix in the bases $v_p$ and $w_q$. We thus obtain a family $(a_\gamma)_{\gamma\in\mathcal{T}}$ of elements of $A^\times$ which we call \defin{non-commutative $\mathcal{A}$-coordinates} of $\mathcal{L}$. The name ``coordinates'' is a slight abuse, since they are not independent.
	
	\begin{proposition}
		For every oriented triangle $(\gamma_1,\gamma_2,\gamma_3)$ of $\mathcal{T}$, we have
		\begin{equation}\label{eq:triangle_relation}
			a_{\gamma_3}a_{\bar\gamma_2}^{-1}a_{\gamma_1} = a_{\bar\gamma_1}a_{\gamma_2}^{-1}a_{\bar\gamma_3}
		\end{equation}
	\end{proposition}

	The coordinates of a decorated twisted $\GL_2(A)$-local system are the holonomies of its abelianized system $\mathcal{E}$ along the lifts of the arcs $\tau_\gamma$, $\gamma\in\mathcal{T}$: for each puncture $p\in P$, the two lifts of $p$ to $\Sigma$ are a sink $p_1$ and a source $p_2$. In the neighborhood of $p_1$, the bundle $\mathcal{E}$ is the pullback of $L_p$ and in the neighborhood of $p_2$ the bundle is the pullback of $\mathcal{L}/L_p$. Now for an arc $\gamma\in\mathcal{T}$ from $p\in P$ to $q\in P$, the lifts of $\tau_\gamma$ to $\Sigma$ join a sink and a source. Denote $\gamma_1$ the lift from $p_1$ to $q_2$ and $\gamma_2$ the lift from $p_2$ to $q_1$. Then $a_\gamma$ is the holonomy of $\mathcal{E}$ along $T'\tau_{\gamma_1}$ and $a_{\bar\gamma}$ is the holonomy of $\mathcal{E}$ along $T'\bar{\tau_{\gamma_2}}$.
	
	Let $\Gamma$ be the graph embedded in $T'\Sigma$ with vertices the lifts of peripheral curves and edges the arcs $\tau_\gamma$, $\gamma\in\mathcal{T^*}$ oriented from sink to source. To each oriented edge of this graph a coordinate is associated, and we assign to the edges with reversed orientation the inverse of this coordinate. The vertices of this graph are curves around which the monodromy of $\mathcal{E}$ is trivial because $\mathcal{L}$ is decorated, so given a path on $\Gamma$, its holonomy in $\mathcal{E}$ is well-defined. Then the triangle relation \eqref{eq:triangle_relation} imply that the monodromy of the abelianized system $\mathcal{E}$ restricted to the graph $\Gamma$ is trivial around every hexagonal tile.
	
	
	Let $\mathcal{T}_1$ and $\mathcal{T}_2$ be two triangulations differing only by one flip. Let $p_1,p_2,p_3,p_4$ be the four (not necessarily distinct) punctures at the vertices of the quadrilateral supporting the flip, in the cyclic order such that $\mathcal{T}_1\setminus \mathcal{T}_2 = \left\lbrace \gamma_{1,3},\gamma_{3,1}\right\rbrace $ and $\mathcal{T}_2\setminus \mathcal{T}_1 = \left\lbrace \gamma_{2,4},\gamma_{4,2}\right\rbrace $ where $\gamma_{i,j}$ is the arc of the quadrilateral going from $p_j$ to $p_i$. Using the path-lifting map, we can compute the relations between the $\mathcal{A}$-coordinates associated to $\mathcal{T}_1$ and the $\mathcal{A}$-coordinates associated to $\mathcal{T}_2$.
	
	\begin{proposition}\label{prop:flip}
		Let $\mathcal{L}$ a decorated twisted $\GL_2(A)$-local system that is both $\mathcal{T}_1$-transverse and $\mathcal{T}_2$-transverse. Then its $\mathcal{A}$-coordinates with respect to $\mathcal{T}_1$ and $\mathcal{T}_2$ satisfy the following exchange relations:
		\[ a_{\gamma_{2,4}} = a_{\gamma_{2,1}}a_{\gamma_{3,1}}^{-1}a_{\gamma_{3,4}} +  a_{\gamma_{2,3}}a_{\gamma_{1,3}}^{-1}a_{\gamma_{1,4}} \]
		\[ a_{\gamma_{4,2}} = a_{\gamma_{4,1}}a_{\gamma_{3,1}}^{-1}a_{\gamma_{3,2}} +  a_{\gamma_{4,3}}a_{\gamma_{1,3}}^{-1}a_{\gamma_{1,2}} \]
	\end{proposition}
	
	\begin{proof}
		For $i\in\left\lbrace 1,2,3,4\right\rbrace $, let $p'_i,p''_i$ be the two lifts of $p_i$ to $\Sigma$ where $p'_i$ is the sink and $p''_i$ is the source. Let $s\in I_{\mathcal{T}_1}(S)\cap I_{\mathcal{T}_2}(S)$ be the intersection of $T'\beta_{p_2}$ and $\gamma_{2,1}$ and let $t\in I_{\mathcal{T}_1}(S)\cap I_{\mathcal{T}_2}(S)$ be the intersection of $T'\beta_{p_4}$ and $\gamma_{3,4}$. Let $\delta$ be a path in $T'S$ from $s$ to $t$ as in figure \ref{fig:flip1}.
		\begin{figure}[h!]
			\includegraphics[scale = 0.5]{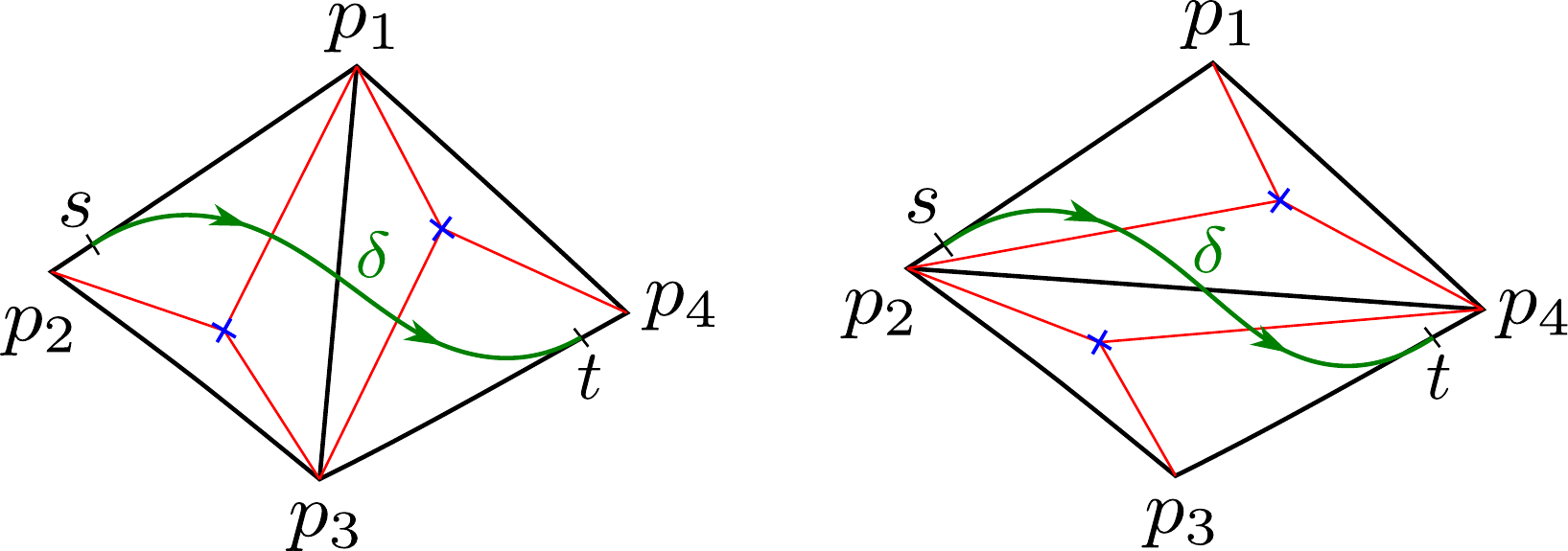}
			\caption{The path $\delta$ on $S$ with triangulation $\mathcal{T}_1$ on the left and with triangulation $\mathcal{T}_2$ on the right.}
			\label{fig:flip1}
		\end{figure}
		
		The holonomy of $\mathcal{L}$ along $\gamma$ does not depend on the triangulation.
		
		Let $\Sigma_1$ and $\mathcal{W}_1$ be the ramified covering and the spectral network associated to the triangulation $\mathcal{T}_1$ and $\Sigma_2,\mathcal{W}_2$ the ones associated to $\mathcal{T}_2$. The corresponding path-lifting maps will be denoted $SN_1$ and $SN_2$, and the corresponding abelianizations of $S$ will be denoted $\mathcal{E}_1$ and $\mathcal{E}_2$.
		
		First, let's lift $\delta$ to $\Sigma_2$ using $SN_2$. Let $s_1,s_2$ the lifts of $s$ to $\Sigma_2$, $s_1$ being the sink and $s_2$ the source. Similarly, let $t_1,t_2$ the lifts of $t$, $t_1$ being the sink and $t_2$ the source. We get \[ SN_2(\delta) =  \delta_1 + \delta_2 + \delta'_1 + \delta'_2 + \delta'_3\]
		where $\delta_1$ is a standard lift from $s_1$ to $t_2$, $\delta_2$ is a standard lift from $s_2$ to $t_1$, $\delta'_1$ is a spectral lift from $s_2$ to $t_2$, $\delta'_2$ is a spectral lift from $s_2$ to $t_1$ and $\delta'_3$ is a spectral lift from $s_1$ to $t_1$ (see Figure~\ref{fig:flip2}). The path $\delta_1$ is the only lift going from $s_1$ to $t_2$, and its holonomy in $\mathcal{E}_2$ in the corresponding bases is $a_{\gamma_{4,2}}$ since it is homotopic to $\tau_{\gamma_{4,2}}$ precomposed with a piece of $\beta_{p_2}$ and postcomposed with a piece of $\beta_{p_4}$, both of which have trivial holonomies. Since $\mathcal{L}$ is the non-abelianization of $\mathcal{E}_2$, this means that the map $L_{p_2}\to\mathcal{L}_{p_4}/L_{p_4}$ obtained by trivializing $\mathcal{L}$ along $\delta$ is exactly $a_{\gamma_{4,2}}$.
		
		\begin{figure}[h!]
			\includegraphics[scale=0.7]{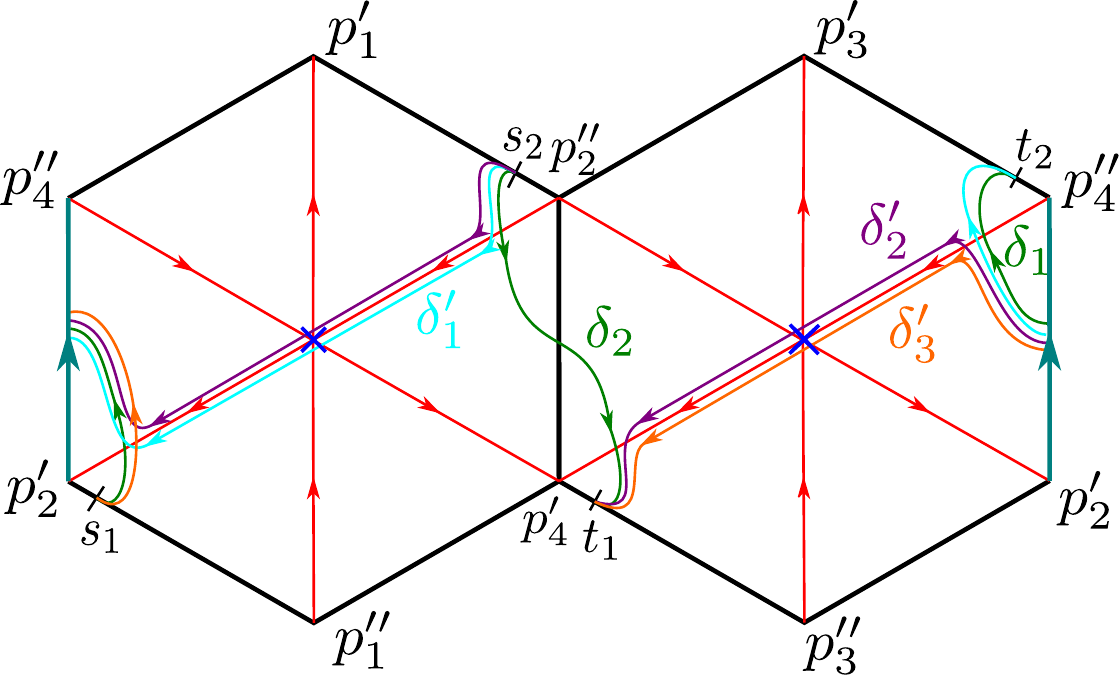}
			\caption{All the lifts of $\delta$ to $\Sigma_2$ using $SN_2$.}
			\label{fig:flip2}
		\end{figure}
		
		Now we will lift $\delta$ to $\Sigma_1$ using $SN_1$. We will keep the same notations as in the previous paragraph. We get \[ SN_1(\delta) =  \delta_1 + \delta_2 + \delta'_1 + \delta'_2 + \delta'_3\]
		where $\delta_1$ is a standard lift from $s_1$ to $t_2$, $\delta_2$ is a standard lift from $s_2$ to $t_1$, $\delta'_1$ is a spectral lift from $s_1$ to $t_1$, $\delta'_2$ is a spectral lift from $s_1$ to $t_2$ and $\delta'_3$ is a spectral lift from $s_2$ to $t_2$ (see Figure~\ref{fig:flip3}). The paths going from $s_1$ to $t_2$ are $\delta_1$ and $\delta'_2$, and their holonomies in $\mathcal{E}_1$ in the corresponding bases are respectively $a_{\gamma_{4,3}}a_{\gamma_{1,3}}^{-1}a_{\gamma_{1,2}}$ and $a_{\gamma_{4,1}}a_{\gamma_{3,1}}^{-1}a_{\gamma_{3,2}}$. These are obtained by retracting the paths on the graph $\Gamma$, as the oriented edges of $\Gamma$ have holonomies given by the $\mathcal{A}$-coordinates. Since $\mathcal{L}$ is also the non-abelianization of $\mathcal{E}_1$, this means that the map $L_{p_2}\to\mathcal{L}_{p_4}/L_{p_4}$ obtained by trivializing $\mathcal{L}$ along $\delta$ must be equal to the holonomy of $\delta_1+\delta'_2$, which give the formula:
		\[ a_{\gamma_{4,2}} = a_{\gamma_{4,1}}a_{\gamma_{3,1}}^{-1}a_{\gamma_{3,2}} +  a_{\gamma_{4,3}}a_{\gamma_{1,3}}^{-1}a_{\gamma_{1,2}} \]
		The formula for $a_{\gamma_{2,4}}$ is obtained similarly.
		\end{proof}
    
    	\begin{figure}[h!]
			\includegraphics[scale=0.7]{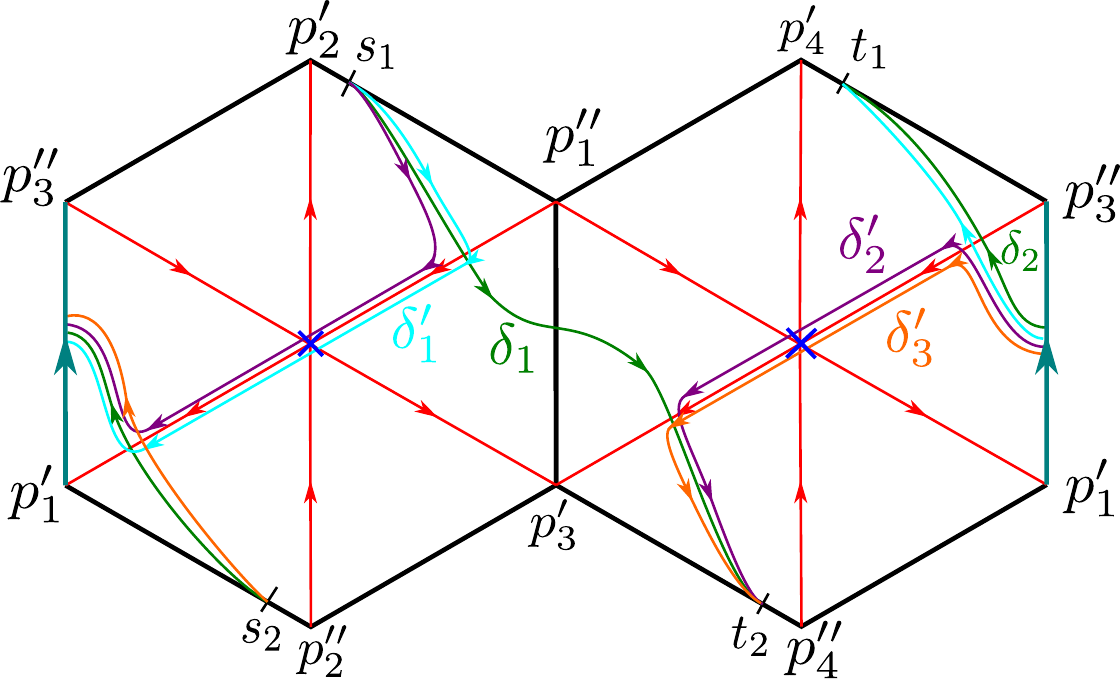}
			\caption{All the lifts of $\delta$ to $\Sigma_1$ using $SN_1$.}
			\label{fig:flip3}
		\end{figure}
	
	This gives a geometric realization of the non-commutative algebra $\mathcal{A}_S$ introduced in \cite{BR}. Using the same type of arguments as above, we can give a topological/geometrical proof of the Laurent phenomenon for the cluster algebra of a polygon:
	
	\begin{theorem}
		Let $n\geq 3$ and let $S_n$ the closed disk with $n$ punctures on the boundary. Let $i,j\in\left\lbrace 1,\dots,n\right\rbrace $, $i\neq j$. Then for every triangulation $\mathcal{T}$ of $S_n$ and every decorated twisted $\GL_2(A)$-local system $\mathcal{L}$ that is both $\mathcal{T}$-transverse and $(i,j)$-transverse, the $\mathcal{A}$-coordinate $a_{\gamma_{i,j}}$ is a non-commutative Laurent polynomial in the $\mathcal{A}$-coordinates $(a_\gamma)_{\gamma\in\mathcal{T}}$ associated to the triangulation $\mathcal{T}$.
	\end{theorem}
	
	\begin{proof}
		All the edges of the form $\gamma_{i,i+1}$, with $i\in P$ ordered cyclically, belong to every triangulation of $S_n$ so the result is immediate. Now let $i,j\in\left\lbrace 1,\dots,n\right\rbrace $, $i\neq j\pm 1$. Let $\mathcal{T}_0$ be a triangulation of $S_n$ containing the edges $\gamma_{i,j}$, $\gamma_{i,j-1}$ and $\gamma_{i-1,j}$. Such a triangulation always exists when $i\neq j\pm 1$. Let $s\in I_{\mathcal{T}_0}(S_n)\cap I_{\mathcal{T}}(S_n)$ be the intersection of $\beta_j$ and $\gamma_{j-1,j}$ and let $t\in I_{\mathcal{T}_0}(S_n)\cap I_{\mathcal{T}}(S_n)$ be the intersection of $\beta_i$ and $\gamma_{i-1,i}$. Let $\delta$ be the path from $s$ to $t$ drawn in Figure~\ref{fig:LP}.
		
		\begin{figure}[h!]
			\includegraphics[scale=0.7]{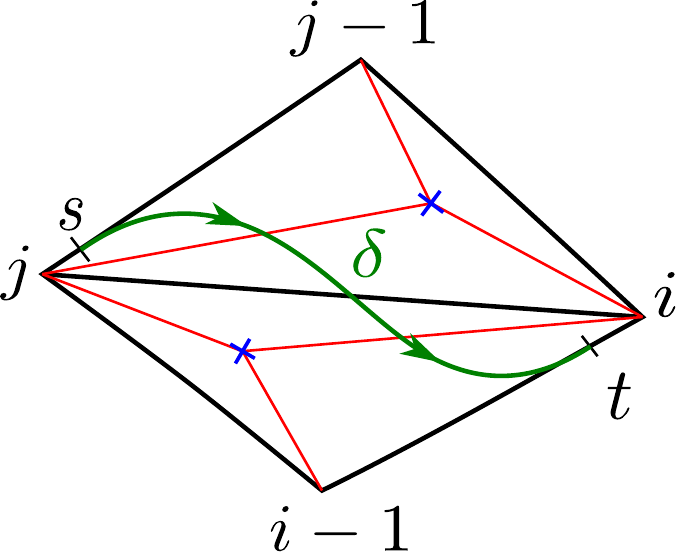}
			\caption{The path $\delta$ in the triangulation $\mathcal{T}_0$. Only the quadrilateral $(i-1,i,j-1,j)$ is drawn.}
			\label{fig:LP}
		\end{figure}
		
		As we have seen in the proof of the flip relation (and keeping the same notations), in the spectral network lift of $\delta$ with respect to the triangulation $\mathcal{T}_0$ the only term from $s_1$ to $t_2$ has the holonomy $a_{\gamma_{i,j}}$ in the abelianization of $\mathcal{L}$ with respect to $\mathcal{T}_0$. This means that the map $L_j\to\mathcal{L}_i/L_i$ obtained by trivializing $\mathcal{L}$ on $\delta$ is $a_{\gamma_{i,j}}$.
		
		Let $\mathcal{E}$ be the abelianization of $\mathcal{L}$ with respect to $\mathcal{T}$. In spectral network lift of $\delta$ with respect to the triangulation $\mathcal{T}$, let $\delta' = \delta'_1+\dots+\delta'_r$ be the sum of all paths from $s_1$ to $t_2$. Each $\delta'_k$ has a holonomy in $\mathcal{E}$ that is a monomial in the coordinates $(a^{\pm 1}_\gamma)_{\gamma\in\mathcal{T}}$ as it retracts on the graph $\Gamma$. Since $\mathcal{L}$ is the non-abelianization of $\mathcal{E}$, the map $L_j\to\mathcal{L}_i/L_i$ obtained by trivializing $\mathcal{L}$ on $\delta$ is equal to the sum of the holonomies of the $\delta'_k$ in $\mathcal{E}$, so it is a Laurent polynomial in the $\mathcal{A}$-coordinates $(a_\gamma)_{\gamma\in\mathcal{T}}$.
	\end{proof}
	
	Using these $\mathcal{A}$-coordinates, we can describe precisely the changes on the $A^\times$-local system on $\Sigma$ induced by a flip in the triangulation. We use the same notations as in Proposition~\ref{prop:flip}. 
	Let $\mathcal{L}$ be a framed twisted $\GL_2(A)$-local system on $S$ that is transverse with respect to both $\mathcal{T}_1$ and $\mathcal{T}_2$. Let $\mathcal{E}_1$ (resp. $\mathcal{E}_2$) be the $A^\times$-local system on $\Sigma$ obtained by abelianizing $\mathcal{L}$ with respect to $\mathcal{T}_1$ (resp. $\mathcal{T}_2$). These changes on the abelianized local system are supported in the lift $C_Q$ of the quadrilateral $Q$ surrounding the flip, which is homeomorphic to a cylinder with four punctures on each boundary components in $\Sigma$. Let $\gamma$ be a loop on $T'\Sigma$. If $\gamma$ only cross one of the two boundary component of $\bar C_Q$ then the monodromies of $\gamma$ in $\mathcal{E}_1$ and $\mathcal{E}_2$ are equal. Suppose $\gamma$ crosses exactly once each of the two boundary components of $\bar C_Q$. Let $\gamma_Q$ be the loop going around $C_Q$ with the same orientation as the boundary of $C_Q$ containing the sinks lifts of $p_2$ and $p_4$ (we refer to this boundary as the \defin{positive} one, and the other one as \defin{negative}).

    \begin{remark}
        We think of the holonomy of $\gamma_Q$ in $\mathcal{E}$ as a generalization in the non-commutative setting of Fock-Goncharov's $\mathcal{X}$-coordinate of the quadrilateral $Q$. If $A=\RR$, the holonomy of $\gamma_Q$ is the cross-ratio of the four lines in $\RR^2$ given by the framing of $\mathcal{L}$.
    \end{remark}

    Up to homotopy, we can assume $\gamma$ is going through at least one point $x_0\in I_{\mathcal{T}^*_1}(\Sigma)\cap I_{\mathcal{T}^*_2}(\Sigma)$ on one of the eight external edges of the hexagon tiling of $Q$. We also choose a representative of $\gamma_Q$ based at $x_0$. Let $b$ be a basis of the fiber of $\mathcal{L}_1$ over $x_0$. Since $x_0$ is not in the interior of the cylinder supporting the flip in $\Sigma$, the fibers of $\mathcal{E}_1$ and $\mathcal{E}_2$ over $x_0$ are the same. Let $Y_1\in A^\times$ (resp. $Y_2$) be the holonomy of $\gamma$ in $\mathcal{E}_1$ (resp. $\mathcal{E}_2$).

    \begin{prop}
        If the part of $\gamma$ inside $C_Q$ goes from the positive boundary to the negative boundary, then 
        $$Y_2 = Y_2(1+X).$$
        If the part of $\gamma$ inside $C_Q$ goes from the negative boundary to the positive boundary, then 
        $$ Y_2 = Y_1(1+X^{-1})^{-1}$$
    \end{prop}

    \begin{remark}
        The element $1+X^{-1}\in A$ is invertible because of the transversality of $\mathcal{L}$ with respect to $\mathcal{T}_2$.
    \end{remark}

	\subsection{Topology of the moduli space of framed twisted local systems}
	
    In this section, we describe the topology of the moduli space of framed twisted $\GL_2(A)$-local systems on $S$ that are transverse to a fixed triangulation $\T$.
	
	As we have seen, framed twisted $\GL_2(A)$-local systems on $S$ that are transverse with respect to a fixed triangulation $\T$ are in 1:1-correspondence with twisted $A^\times$-local systems on $\Sigma$. Since $\Sigma$ has punctures, the space of twisted and non-twisted $A^\times$-local systems are homeomorphic. So we obtain the following theorem, using the same notations as in Proposition~\ref{prop:topology_Sigma}:
	
	\begin{teo}
		The moduli space of framed (twisted) $\GL_2(A)$-local systems on $S$ that are transverse with respect to a fixed triangulation $\T$ is homeomorphic to the moduli space of  (twisted) $A^\times$-local systems on $\Sigma$ which is homeomorphic to
		$(A^\times)^{1-4\chi(\bar S)+2p+\sum n_i}/A^\times$
		where $A^\times$ acts diagonally by conjugation on $(A^\times)^{1-4\chi(\bar S)+2p+\sum n_i}$.
	\end{teo}
	\begin{rem}\label{rk:comp_tec}
	    In~\cite{GRW} the         authors prove the same result using different technics. They define local systems on some appropriate graphs over $S$ and parametrize them using coordinates that are similar to Fock-Goncharov's $\GL_n$-cluster $\mathcal X$-coordinates~\cite{FG}.
	\end{rem}
	
	Since any twisted peripherally unipotent $\GL_2(A)$-local system on $S$ has exactly one framing, we obtain:
	
	\begin{cor}
		The moduli space of twisted peripherally unipotent $\GL_2(A)$-local systems on $S$ whose unique framing is transverse with respect to a fixed triangulation $\T$ is homeomorphic to the moduli space of  twisted $A^\times$-local systems on $\bar\Sigma$.
	\end{cor}
	
	\begin{cor}
		The moduli space of decorated twisted peripherally unipotent $\GL_2(A)$-local systems on $S$ that are transverse with respect to a fixed triangulation $\T$ is homeomorphic to the product of the moduli space of twisted $A^\times$-local systems on $\bar\Sigma$ and $(A^\times)^{p}$.
	\end{cor}
	
	\section{Symplectic groups over involutive algebras ans symplectic local systems}\label{sec:symplectic_case}
	
	Involutive algebras are an important class of non-commutative algebras. Over involutive algebras, generalizations of many classical groups can be constructed (e.g. orthogonal groups, symplectic groups). In this chapter, we define algebras with anti-involutions and symplectic groups over such algebras that were introduced and studied in~\cite{ABRRW}. Further, we introduce framed twisted symplectic local system and characterize them in terms of partial abelianization introduced before.
	
	\subsection{Involutive algebras}
	
	Let $A$ be a unital associative, possibly non-commutative 
	$\R$-algebra.
	
	\begin{df}\label{def:antiinv}
		An \defin{anti-involution} on $A$ is a $\R$-linear map $\sigma\colon A\to A$ such that
		\begin{itemize}
			\item $\sigma(ab)=\sigma(b)\sigma(a)$;
			\item $\sigma^2=\Id$.
		\end{itemize}
		An \defin{involutive $\R$-algebra} is a pair $(A,\sigma)$, where $A$ is a $\R$-algebra and $\sigma$ is an anti-involution on $A$.
	\end{df}
	
	\begin{df}
		Two elements $a,a'\in A$ are called \defin{congruent}, if there exists $b\in A^\times$ such that $a'=\sigma(b)ab$.
	\end{df}
	
	\begin{df} An element $a\in A$ is called \defin{$\sigma$-symmetric} if $\sigma(a)=a$. An element $a\in A$ is called \defin{$\sigma$-anti-symmetric} if $\sigma(a)=-a$. We denote
		$$A^{\sigma}:=\Fix_A(\sigma)=\{a\in A\mid \sigma(a)=a\},$$
		$$A^{-\sigma}:=\Fix_A(-\sigma)=\{a\in A\mid \sigma(a)=-a\}.$$
	\end{df}
	
	\begin{df}
		The closed subgroup $$U_{(A,\sigma)}=\{a\in A^\times\mid \sigma(a)a=1\}$$ of $A^\times$ is called the \defin{unitary group} of $A$. The Lie algebra of $U_{(A,\sigma)}$ agrees with $A^{-\sigma}$.
	\end{df}
	
	\begin{df}\label{df:cone}
		Let $(A,\sigma)$ be an $\R$-algebra with an anti-involution. We define two set of squares:
		$$A^\sigma_+:=\left\{a^2\mid a\in (A^\sigma)^\times\right\},\;A^\sigma_{\geq 0}:=\left\{a^2\mid a\in A^\sigma \right\}.$$
	\end{df}
	
	\begin{rem}
		Since the algebra $A$ is unital, we always have the canonical copy of $\R$ in $A$, namely $\R\cdot 1$ where $1$ is the unit of $A$. We will always identify $\R\cdot 1$ with $\R$. Moreover, since $\sigma$ is linear, for all $k\in \R$, $\sigma(k\cdot 1)=k\sigma(1)=k\cdot 1$, i.e. $\R\cdot 1\subseteq A^\sigma$ and $\R_{>0}\cdot 1\subseteq A^\sigma_+$.
	\end{rem}
	
	\begin{df}\label{df:Herm_alg}
		A unital associative finite dimensional $\R$-algebra with an anti-involution $(A,\sigma)$ is called \defin{Hermitian} if for all $x,y\in A^\sigma$, $x^2+y^2=0$ implies $x=y=0$.
	\end{df}
 
	\begin{rem}
		In~\cite{ABRRW}, the property to be Hermitian is defined in the same way for algebras with an anti-involution over any real closed field. In this paper, we are discussing only Hermitian algebras over $\R$. \end{rem}
	
	\begin{rem}\label{rk:pc_cone}
		In~\cite{ABRRW} is shown that, if $(A,\sigma)$ is a Hermitian algebra, then $A^\sigma_+$ is an open proper convex cone in $A^\sigma$, where proper means that the set does not contain (affine) lines.
	\end{rem}	

    If $(A,\sigma)$ is Hermitian, for an element $a\in A^\sigma$ the \defin{signature} can be defined, which is a bounded function $\sgn\colon A^\sigma\to \Z$ that is invariant under congruence by elements of $A^\times$. The elements of maximal signature are precisely the elements of $A^\sigma_+$. For more details about the signature see~\cite{ABRRW}.
	
	\subsection{Symplectic groups over non-commutative algebras}\label{sec:symplectic}
	
	Let $A$ be a unital associative finite dimensional $\R$-algebra with an anti-involution $\sigma$. We consider $A^2$ as a right $A$-module over $A$.
	
	\begin{df}
		Let $\omega(x,y):=\sigma(x)^T\Omega y$ with $\Omega=\Ome{1}$. The group
		$$\Sp_2(A,\sigma):=\Aut(\omega)=\{g\in M_2(A)\mid \sigma(g)^T\omega g=\omega\}$$
		is the \defin{symplectic group} $\Sp_2$ over $(A,\sigma)$. The form $\omega$ is called the \defin{standard symplectic form} on $A^2$.
	\end{df}
	
	We have
	$$\Sp_2(A,\sigma)=\left\{\begin{pmatrix}
		a & b \\
		c & d
	\end{pmatrix}\mid \sigma(a)c,\,\sigma(b)d\in A^\sigma,\,\sigma(a)d-\sigma(c)b=1\right\}\subseteq \GL_2(A)$$
	We can also determine the Lie algebra $\spp_2(A,\sigma)$ of $\Sp_2(A,\sigma)$:
	$$\spp_2(A,\sigma)=\left\{\begin{pmatrix}
		x & z \\
		y & -\sigma(x)
	\end{pmatrix}\mid x\in A,\;y,z\in A^\sigma\right\}\subseteq M_2(A).$$
	
	\begin{rem}
		In~\cite{ABRRW} is shown that, if $A$ is a Hermitian algebra, then $\Sp_2(A,\sigma)$ is a Hermitian Lie group of tube type.
	\end{rem}
	
	Let $(x,y)$ be a basis of $A^2$. We say that this basis is \defin{isotropic} if $\omega(x,x)=\omega(y,y)=0$. We say that this basis is \defin{symplectic} if furthermore $\omega(x,y)=1$.
	
	Let $x\in A^2$ be a regular isotropic element, i.e. $\omega(x,x)=0$. We call the set $xA:=\{xa\mid a\in A\}$ an \defin{isotropic $A$-line}. The space of all isotropic $A$-lines is denoted by $\Is(\omega)$.
	
	\subsection{Symplectic local systems}
	
	We consider a twisted $\GL_2(A)$-local system $\mathcal L\to T'S$ over $S$. We say that $\mathcal L$ is a \defin{twisted $\Sp_2(A,\sigma)$-local system} (or just \defin{twisted symplectic local system}) if there exists a parallel field of the standard symplectic 2-form $\omega\colon\mathcal L\times\mathcal L\to A$ on $T'S$. We say that $\mathcal L$ is \defin{peripherally parabolic (or unipotent)} if it is parabolic (resp. unipotent) as a twisted $\GL_2(A)$-local  system.
	
	A framing of a parabolic twisted symplectic local system is called \defin{isotropic} if the parallel subbundle defining the framing in a neighborhood of every puncture is isotropic with respect to the field of the form $\omega$. A decoration $((v_p)_{p\in P},(w_p)_{p\in P})$ of a unipotent twisted symplectic local system is called \defin{symplectic} if $\omega(v_p,v_p)=0$ and $\omega(v_p,w_p)=1$.
	
	\begin{rem} 
	    Notice, that if $\omega(v_p,v_p)=0$, then the expression $\omega(v_p,w_p)$ is well-defined. Indeed, let $\tilde w_p$ and $\tilde w'_p$ be two lifts of $w_p$ to $A^2$. Then $\tilde w'_p=\tilde w_p+v_p a$ for some $a\in A$. Further, 
	    $$\omega(v_p,\tilde w'_p)=\omega(v_p,\tilde w_p+v_pa)=\omega(v_p,\tilde w_p)=:\omega(v_p,w_p).$$
	    It is always enough to choose $v_p$ for every $p\in P$. Then $w_p$ becomes uniquely defined.
	\end{rem}
	
	A \defin{framed twisted symplectic local system} is a peripherally parabolic twisted symplectic local system with an isotropic framing. A \defin{decorated twisted symplectic local system} is a peripherally unipotent twisted symplectic local system with a symplectic decoration.
	
	\begin{rem}
		Notice, that since $\omega$ is a parallel form of even degree, the parallel transport of $\omega$ around the fiber of $T'S$ is trivial.
	\end{rem}
	
	Let $\pi\colon\Sigma\to S$ be the ramified two-fold covering as before. Let $\mathcal{E}\to T'\Sigma$ be an $A^\times$-local system over the spectral covering $\Sigma$ of $S$ that is obtained by the partial abelianization procedure.
	
	Let $\theta\colon\Sigma\to\Sigma$ be the covering involution. Slightly abusing the notation, we also denote $\theta=\theta_*\colon T'\Sigma\to T'\Sigma$.
	
	\begin{rem}
		Notice that $\theta$ does not have fixed points in $T'\Sigma$.
	\end{rem}
	
	We consider the pull-back of $\mathcal{E}$ with respect to $\theta$ and denote it by $\mathcal{E}':=\theta^*\mathcal{E}$.
	To simplify the notation, we will identify $\mathcal{E}'_p$ and $\mathcal{E}_{\theta(p)}$ for all $p\in\Sigma$.
	We denote by $P_\gamma\colon \mathcal{E}_{\gamma(0)}\to \mathcal{E}_{\gamma(1)}$, $P'_\gamma=P_{\theta\circ\gamma}\colon \mathcal{E}_{\theta(\gamma(0))}\to \mathcal{E}_{\theta(\gamma(1))}$ the parallel transport along $\gamma\colon[0,1]\to\Sigma$ in $\mathcal{E}$ and $\mathcal{E}'$. We denote by $P^S_\alpha\colon V_{\alpha(0)}\to V_{\alpha(1)}$ the parallel transport along $\alpha\colon[0,1]\to S$ in $\mathcal L$.
	
	\begin{df}
		Let $V$ and $V'$ be two right $A$-modules. A map $b\colon V\times V'\to A$ is called an \defin{$A$-sesquilinear pairing} between $V$ and $V'$ if it is additive in every argument and if for all $v\in V$, $v'\in V'$, and for all $a,a'\in A$, $b(va,v'a')=\sigma(a)b(v,v')a$. An $A$-sesquilinear paring $b$ is \defin{non-degenerate} if for every regular $v\in V$ there exists $v'\in V'$ such that $b(v,v')\in A^\times$ and for every regular $v'\in V'$ there exists $v\in V$ such that $b(v,v')\in A^\times$.
	\end{df}
	
	We denote by $B(\mathcal{E},\mathcal{E}')\to T'\Sigma$ the vector bundle of all $A$-sesquilinear parings between $\mathcal{E}$ and $\mathcal{E}'$. A section $\beta\in\Gamma(T'\Sigma,B(\mathcal{E},\mathcal{E}'))$ is called \defin{parallel} if $$\beta_{\gamma(0)}(x,y)=\beta_{\gamma(1)}(P_\gamma(x),P'_\gamma(y))=\beta_{\gamma(1)}(P_\gamma(x),P_{\theta\circ\gamma}(y))$$
	for every $\gamma\colon[0,1] \to T'\Sigma$ and for every $x\in \mathcal{E}_{\gamma(0)}$, $y\in \mathcal{E}'_{\gamma(0)}=\mathcal{E}_{\theta(\gamma(0))}$.
	
	\begin{rem}
		Notice that if $\beta\in\Gamma(T'\Sigma,B(\mathcal{E},\mathcal{E}'))$ is parallel and $\beta_p$ is non-degenerate for one $p\in T'\Sigma$, then $\beta_p$ is non-degenerate for all $p\in T'\Sigma$.
	\end{rem}
	
	\begin{teo}\label{symplectic}
		The framed local system $\mathcal L$ is an $\Sp_2(A,\sigma)$-local system if and only if there exists a non-degenerate parallel section $\beta\in\Gamma(T'\Sigma,B(\mathcal{E},\mathcal{E}'))$ such that $\beta_{p}(x,y)=-\sigma(\beta_{\theta(p)}(y,x))$ for every $p\in T'\Sigma$, for every $x\in \mathcal{E}_p$ and for every $y\in \mathcal{E}_{\theta(p)}$.
	\end{teo}
	
	\begin{proof} $(\Rightarrow)$
		Assume, $\mathcal L$ is an $\Sp_2(A,\sigma)$-local system. That means, there exists a field of standard symplectic forms $\omega$ on $\mathcal L\to T'S$, such that for every $\alpha\colon [0,1]\to T'S$ and for every $v,w\in V_{\alpha(0)}$,
		$$\omega_{\alpha(0)}(v,w)=\omega_{\alpha(1)}(P^S_\alpha(v),P^S_\alpha(w)).$$
		
		Let $\gamma\colon [0,1]\to T'\Sigma$ be a smooth path such that $\gamma(0),\gamma(1)$ do not project to points on lines of the spectral network on $\Sigma$, and $x\in \mathcal{E}_{\gamma(0)}$ and $y\in \mathcal{E}_{\theta(\gamma(0))}$ regular elements. We consider $\gamma'=\theta\circ\gamma$ and $\alpha=\pi\circ\gamma=\pi\circ\gamma'$.  Moreover, $(\pi_*(x),\pi_*(y))$ is an isotropic basis of $\mathcal L_{\alpha(0)}$. We can define
		$$\beta_{\gamma(0)}(x,y):=\omega_{\alpha(0)}(\pi_*(x),\pi_*(y)).$$
		Since $\omega$ is non-degenerate and skew-Hermitian, $\beta$ is non-degenerate and sesquilinear pairing. Moreover, $\beta_{\gamma(0)}(x,y)=-\sigma(\beta_{\theta(\gamma(0))}(y,x))$ because $\omega_{\alpha(0)}(\pi_*(x),\pi_*(y))=-\sigma(\omega_{\alpha(0)}(\pi_*(y),\pi_*(x)))$.
		
		If $\gamma$ does not intersect lines of the spectral network, then $\beta$ along $\gamma$ is parallel because in this case $P^S_{\alpha}=P_{\gamma}\oplus P_{\sigma\circ\gamma}$
		
		If $\gamma$ is a small segment intersecting a line of spectral network, then
		$$\omega_{\alpha(1)}(P^S_\alpha(\pi_*(x)),P^S_\alpha(\pi_*(y)))= \omega_{\alpha(1)}(\pi_*(P_\gamma(x))+\pi_*(P_{\tilde\gamma}(x)),\pi_*(P_{\theta\circ\gamma}(y)))$$
		where $\tilde\gamma$ is a lift of $\alpha$ going along a line of spectral network from $\gamma(0)$ to $\theta(\gamma(1))$.
		But elements $P_{\tilde\gamma}(x),P_{\theta\circ\gamma}(y)\in \mathcal{E}_{\theta(\gamma(1))}$, therefore, $\omega(\pi_*(P_{\tilde\gamma}(x)),\pi_*(P_{\theta\circ\gamma}(y)))=0$. So
		\begin{align*}
			\beta_{\gamma(0)}(x,y) & =\omega_{\alpha(0)}(\pi_*(x),\pi_*(y))\\
			& =\omega_{\alpha(1)}(P^S_\alpha(\pi_*(x)),P^S_\alpha(\pi_*(y)))\\
			& =\omega_{\alpha(1)}(\pi_*(P_\gamma(x)),\pi_*(P_{\theta\circ\gamma}(y)))\\
			& =\beta_{\alpha(1)}(P_\gamma(x),P_{\theta\circ\gamma}(y)),
		\end{align*}
		i.e. $\beta$ is parallel and extends also along lines of the spectral network on $\Sigma$.
		
		Finally, let $p\in T'\Sigma$. Let $x\in \mathcal{E}_p$ and $y\in \mathcal{E}_{\theta(p)}$ regular elements. Then $(\pi_*(x),\pi_*(y))$ is an isotropic basis of $\mathcal L_{\pi(p)}$, i.e. $\beta(x,y)=\omega(\pi_*(x),\pi_*(y))\in A^\times$. So the pairing $\beta$ is non-degenerate.
		
		$(\Leftarrow)$ Assume, there exists a non-degenerate parallel sesquilinear pairing $\beta$. Let $p\in T'\Sigma$ that does not project to a point on a line of the spectral network on $\Sigma$. We define for every $x\in \mathcal{E}_p, y\in \mathcal{E}_{\theta(p)}$:
		$$\omega_{\pi(p)}(\pi_*(x),\pi_*(y)):=\beta_p(x,y).$$
		Because $(\pi_*(x),\pi_*(y))$ is a basis of $V_{\pi(p)}$, $\omega$ extends by sesquilinearity on $V_{\pi(p)}$ if we assume $$\omega_{\pi(p)}(\pi_*(x),\pi_*(x'))=\omega_{\pi(p)}(\pi_*(y),\pi_*(y;))=0$$
		for all $x,x'\in \mathcal{E}_p$ and $y,y'\in \mathcal{E}_{\theta(p)}$. Since $\beta$ is non-degenerate, $\omega$ is non-degenerate as well.

		Since $\beta_{p}(x,y)=-\sigma(\beta_{\theta(p)}(y,x))$, we get
		$$\omega_{\pi(p)}(\pi_*(y),\pi_*(x))=\beta_{\theta(p)}(y,x)=-\sigma(\beta_p(x,y))=-\sigma(\omega_{\pi(p)}(\pi_*(x),\pi_*(y))).$$
		
		Further, $\omega$ is parallel. Indeed, let $\alpha\colon [0,1]\to T'S$ be a path such that the projections of $\alpha(0)$ and $\alpha(1)$ to $S$ are not on the lines of the spectral network. Let $x,y\in \mathcal L_{\alpha(0)}$. Let $\alpha_1$, $\alpha_2:=\theta\circ\alpha_1$ are two standard lifts of $\alpha$ to $T^1\Sigma$. Then $x=\pi_*(x_1)+\pi_*(x_2)$ and $y=\pi_*(y_1)+\pi_*(y_2)$ where $x_1,y_1\in \mathcal{E}_{\alpha_1(0)}$ and $x_2,y_2\in \mathcal{E}_{\alpha_2(0)}$. If the projection of $\alpha$ to $\Sigma$ does not intersect the spectral network, then the projection $T'\Sigma\to T'S$ and the parallel transport along $\alpha$ and $\alpha_1$, $\alpha_2$ commute. So $\omega$ is parallel because $\beta$ is parallel.
		
		Assume now that the projection of $\alpha$ intersects the spectral network once. We denote by $\alpha_3$ the additional lift of $\alpha$ along the spectral network. Without loss of generality, assume $\alpha_3(0)=\alpha_1(0)$ and $\alpha_3(1)=\alpha_2(1)$. Notice that the path $\theta\circ(\alpha_3.\bar \alpha_1).\alpha_3.\bar \alpha_1$ is homotopic to the fiber of $T'\Sigma\to \Sigma$. Therefore, $P_{\theta\circ\bar\alpha_1.\alpha_3}=-P_{\theta\circ\bar\alpha_3.\alpha_1}$. Therefore,
		\begin{align*}
			\omega_{\alpha(1)}(P^S_\alpha(x),P^S_\alpha(y))& =  \omega_{\alpha(1)}(P^S_\alpha(x),P^S_\alpha(y))\\
			& =  \omega_{\alpha(1)}(P^S_\alpha(\pi_*(x_1))+P^S_\alpha(\pi_*(x_2)),P^S_\alpha(\pi_*(y_1))+P^S_\alpha(\pi_*(y_2)))\\
			& =  \omega_{\alpha(1)}(\pi_*(P_{\alpha_1}(x_1)+P_{\alpha_3}(x_1)+P_{\alpha_2}(x_2)),\pi_*(P_{\alpha_1}(y_1)+P_{\alpha_3}(y_1)+P_{\alpha_2}(y_2)))\\
			& =  \omega_{\alpha(1)}(\pi_*(P_{\alpha_1}(x_1),\pi_*(P_{\alpha_3}(y_1)+P_{\alpha_2}(y_2))))\\
			& +\omega_{\alpha(1)}(\pi_*(P_{\alpha_3}(x_1) + P_{\alpha_2}(x_2)),\pi_*(P_{\alpha_1}(y_1)))\\
			& =  \beta_{\alpha_1(1)}(P_{\alpha_1}(x_1),P_{\alpha_3}(y_1)+P_{\alpha_2}(y_2))+\beta_{\alpha_2(1)}(P_{\alpha_3}(x_1)\\
			& + P_{\alpha_2}(x_2),P_{\alpha_1}(y_1))\\
			& =  \beta_{\alpha_1(1)}(P_{\alpha_1}(x_1),P_{\alpha_3}(y_1))+\beta_{\alpha_1(1)}(P_{\alpha_1}(x_1),P_{\alpha_2}(y_2))\\
			& + \beta_{\alpha_2(1)}(P_{\alpha_3}(x_1),P_{\alpha_1}(y_1)) + \beta_{\alpha_2(1)}(P_{\alpha_2}(x_2),P_{\alpha_1}(y_1))\\
			& =  \beta_{\alpha_1(1)}(P_{\alpha_1}(x_1),P_{\alpha_2}(y_2))+\beta_{\alpha_2(1)}(P_{\alpha_2}(x_2),P_{\alpha_1}(y_1))\\
			& + \beta_{\alpha_1(0)}(x_1,P_{(\theta\circ\bar\alpha_3).\alpha_1}(y_1) +P_{(\theta\circ\bar\alpha_1).\alpha_3}(y_1))\\
			& =  \beta_{\alpha_1(1)}(P_{\alpha_1}(x_1),P_{\alpha_2}(y_2))+\beta_{\alpha_2(1)}(P_{\alpha_2}(x_2),P_{\alpha_1}(y_1))\\
			& + \beta_{\alpha_1(0)}(x_1,P_{\theta\circ\bar\alpha_3.\alpha_1}(y_1) +P_{\theta\circ\bar\alpha_1.\alpha_3}(y_1))\\
			& =  \beta_{\alpha_1(1)}(P_{\alpha_1}(x_1),P_{\alpha_2}(y_2))+\beta_{\alpha_2(1)}(P_{\alpha_2}(x_2),P_{\alpha_1}(y_1))\\
			& =  \beta_{\alpha_1(0)}(x_1,y_2)+\beta_{\alpha_2(0)}(x_2,y_1)\\
			& =  \omega_{\alpha(0)}(x,y).
		\end{align*}
		So $\omega$ is parallel and extends also along lines of the spectral network on $S$.
		
		Finally, let $p\in\Sigma$ and $x\in \mathcal{E}_p$, $y\in \mathcal{E}_{\theta(p)}$ such that $\beta_{p}(x,y)=1$, then $\omega(\pi_*(x),\pi_*(y))=1$. So $\omega$ is a field of standard symplectic forms.
	\end{proof}
	
	\subsection{Topology of the moduli space of framed twisted symplectic local systems}
	
	We keep the same notations as in Proposition~\ref{prop:topology_Sigma}. Our goal in this section is to prove the following theorem:
	
	\begin{teo}\label{sympectic_local_systems}
		The moduli space of framed (twisted) $\Sp_2(A,\sigma)$-local systems on $S$ that are transverse with respect to a fixed triangulation $\T$ is homeomorphic to:
		$$\left(((A^\sigma)^\times)^{-2\chi(\bar S)+2p-1+\sum n_i}\times (A^\times)^{1-\chi(\bar S)+p}\right)/A^\times$$
		where the group $A^\times$ acts componentwisely by conjugation on $(A^\times)^{1-\chi(\bar S)+p}$ and by congruence on $((A^\sigma)^\times)^{-2\chi(\bar S)+2p-1+\sum n_i}$.
	\end{teo}
	
	\begin{proof}
		We use the 1:1-correspondence between framed twisted $\Sp_2(A,\sigma)$-local systems on $S$ that are transverse to a fixed triangulation $\T$ and twisted $A^\times$-local systems on $\Sigma$ equipped with a non-degenerate parallel pairing $\beta$ as in Theorem~\ref{symplectic}.
		
		Let $\tilde b\in T'\Sigma$ such that it projects to a ramification point $b\in\Sigma$. Let $\alpha_1,\dots,\alpha_s\colon [0,1]\to S$ are free generators of the fundamental group $\pi_1(S,\pi(b))$. Let $\gamma_i^1,\gamma_i^2$ are closed lifts of $\alpha_i$ to $T'\Sigma$ such that $\theta\circ\gamma_i^1=\gamma_i^2$ and $\gamma_i^1$ is based at $\tilde b$. Notice, that then $\gamma_i^2$ is based at $\theta(\tilde b)$.
		
		Let $s^+_{\tilde b}$ be as before the path from $\tilde b$ to $\theta(\tilde b)$ going along the fiber at $b$ in the positive direction and $s^-_{\theta(\tilde b)}:=\overline{s^+_{\tilde b}}$ the path from $\theta(\tilde b)$ to $\tilde b$ going along the fiber at $b$ in the negative direction. If the context is clear, we just write $s^+$ or $s^-$ to simplify the notation. 
		
		Let $x\in \mathcal{E}_{\tilde b}$. Then on one hand:
		$\beta_{\tilde b}(x,P_{s^+}(x))=-\sigma(\beta_{\theta(\tilde b)}(P_{s^+}(x),x)).$
		On the other hand, since $\beta$ is parallel:
		\begin{align*}
			\beta_{\tilde b}(x,P_{s^+}(x)) =\;& \beta_{\theta(\tilde b)}(P_{s^+}(x),P_{s^+}(P_{s^+}(x)))\\
			=\;& \beta_{\theta(\tilde b)}(P_{s^+}(x),-x)\\
			=\;& -\beta_{\theta(\tilde b)}(P_{s^+}(x),x).
		\end{align*}
		So we obtain:
		$$\beta_{\tilde b}(x,P_{s^+}(x))=-\beta_{\theta(\tilde b)}(P_{s^+}(x),x)=-\sigma(\beta_{\theta(\tilde b)}(P_{s^+}(x),x))=:a_0\in A^\sigma.$$
		Let now $\gamma$ be a loop based at $\tilde b$ and
		$$a_0=\beta_{\tilde b}(x,P_{s^+}(x))=\beta_{\tilde b}(P_\gamma(x),P_{\theta\circ\gamma}P_{s^+}(x)).$$
		For every $x\in \mathcal{E}_{\tilde b}$, $P_\gamma(x)=xa_\gamma$ where $a_\gamma\in A^\times$. Let $P_{\theta\circ\gamma}P_{s^+}(x)=P_{s^+}(x)a'_\gamma$ for $a'_\gamma\in A^\times$. Then
		$$a_0=\sigma(a_\gamma)\beta_{\tilde b}(x,P_{s^+}(x))a'_\gamma=\sigma(a_\gamma)a_0 a'_\gamma,$$
		$$a'_\gamma=a_0^{-1}\sigma(a_\gamma^{-1})a_0.$$
		
		Let $\gamma$ and $s^-.(\theta\circ\bar\gamma).s^+$ are different generators of $\pi_1(T'\Sigma,\tilde b)$ (this corresponds to curves $\gamma_i^1$ and $\gamma_i^2$ of Lemma~\ref{counting} case~(1) lifted to $T'\Sigma$). In particular, they are not homotopic. Then $a_\gamma$ and $a_0$ determine uniquely $a'_\gamma$.
		
		Let $\gamma\colon[0,1]\to T'\Sigma$ and $\theta\circ\gamma\colon[0,1]\to T'\Sigma$ are two lifts to $T'\Sigma$ of a segment in $S$ connecting $\pi(b)$ and $\pi(b')$ where $b'$ is another ramification point on $\Sigma$. Let $\tilde b:=\gamma(0)$ and $\tilde b':=\gamma(1)$. In this case, $\xi_{\tilde b'}:=\xi:=s^-_{\theta(\tilde b)}.\theta(\bar\gamma).s^+_{\tilde b'}.\gamma$ and $s^-.(\theta\circ\bar\xi).s^+$ are homotopic in $T'\Sigma$. Therefore, $a_\xi=a_0^{-1}\sigma(a_\xi)a_0$, i.e. $a_0a_\xi\in A^\sigma$. Moreover, an easy calculation shows that $a_0a_\xi=\beta_{\tilde b'}(y,P_{s^+_{\tilde b'}}y)$ where $y=P_\gamma(x)$.
		
		So the symplectic local system provides us elements $a_i\in A^\times$ corresponding to $P_{\gamma_i^1}$, $a_0\in A^\sigma$ and $a_0a_\xi\in A^\sigma$ for every $\xi$ as in~(2) of Lemma~\ref{counting} (lifted to $T'\Sigma$). These elements are well-defined up to a common conjugation of all $a_i$ and common congruence of all $a_0$ and $a_0a_\xi$ by an element of $A^\times$.
		
		Conversely, if elements $a_i$, $a_0$, $a_\xi$ as above are given, then a twisted $A^\times$-local systems on $\Sigma$ equipped with a non-degenerate parallel pairing $\beta$ can be reconstructed uniquely. Equivalent local system correspond to a common conjugation of all $a_i$ and common congruence of all $a_0$ and $a_0a_\xi$ by an element of $A^\times$.
	\end{proof}
	
	\subsection{Symplectic local system over Hermitian algebras}
	
	Let $A$ be a Hermitian algebra. Let $\ell_1$, $\ell_2$, $\ell_3$ be pairwise transverse isotropic $A$-lines. The Kashiwara-Maslov index of the triple $(\ell_1,\ell_2,\ell_3)$ is the signature of the element $\omega(x,\mu_1^{23}(x))\in (A^\sigma)^\times$ for a regular $x\in \ell_1$ where $\mu_1^{23}$ is the Kashiwara-Maslov map defined in Section~\ref{sec:Kashiwara}. In fact, this signature does not depend on $x\in \ell_1$, and it is invariant under cyclic permutations of the triple $(\ell_1,\ell_2,\ell_3)$ and it changes the sign by transposition of the elements of the triple.
	
	Let $\tau\subset S$ be a triangle of the triangulation $\mathcal T$ that is incident to punctures $p_1$, $p_2$, $p_3$ and the orientation of the triangle agrees with the orientation of the triple $(p_1,p_2,p_3)$. As in Section~\ref{sec:Kashiwara}, let $L_i$ be a parallel isotropic $A$-subbundle of $\mathcal L\to T'\tau$ corresponding to the puncture $p_i$, $i\in\{1,2,3\}$. Let $H=\pi^{-1}(\tau)\subset\Sigma$ be the hexagon that covers $\tau$. Let $b$ be the ramification point in $H$, let $\tilde b$ be a lift of $b$ in $T'H$ and let $s^+$ be a path in $T'H$ going from $\tilde b$ to $\theta(\tilde b)$ along the fiber in the positive direction.
	
	The following proposition is immediate:
	\begin{prop}
		Let $z\in T'\tau$. The Kashiwara-Maslov index of $(L_1(z),L_2(z),L_3(z))$ agrees with the signature of the element $\beta_{\tilde b}(x,P_{s^+}(x))\in A^\sigma$ for a regular $x\in \mathcal{E}_{\tilde b}$.
	\end{prop}
	
	\begin{teo}
		If $A$ is Hermitian, then the moduli space of framed (twisted) maximal $\Sp_2(A,\sigma)$-local systems on $S$ is homeomorphic to:
		$$\left((A^\sigma_+)^{-2\chi(\bar S)+p}\times (A^\times)^{-2\chi(\bar S)+2p-1+\sum n_i}\right)/A^\times$$
		where $A^\times$ acts componentwisely by conjugation on $(A^\times)^{-2\chi(\bar S)+2p-1+\sum n_i}$ and by congruence on $(A^\sigma_+)^{-2\chi(\bar S)+p}$.
	\end{teo}
	
	\begin{proof}
		Following the notation of the proof of Theorem~\ref{sympectic_local_systems}, notice that the signature of $a_0\in A^\sigma$ agrees with the Kashiwara-Maslov index of the oriented triangle where the ramification point $\pi(b)\in S$ lies, and the signature of $a_0a_{\xi_{\tilde b'}}\in A^\sigma$ agrees with the Kashiwara-Maslov index of the oriented triangle where the ramification point $\pi(b')\in S$ lies. A twisted symplectic local system is maximal if and only if Kashiwara-Maslov indices of all oriented triangles are maximal. So we obtain the statement of the theorem.
	\end{proof}
	
	\begin{rem}
		The results of this and previous sections agree with the results from~\cite{GRW} obtained using different techniques (see also Remark~\ref{rk:comp_tec}).
	\end{rem}
	
	\subsection{\texorpdfstring{$\mathcal A$}{A}-coordinates for symplectic local systems}
	
	Since $\Sp_2(A,\sigma)$ is a subgroup of $\GL_2(A)$, the $\mathcal{A}$-coordinates defined in section \ref{sec:A-coord}, a twisted symplectic local system have well-defined $\mathcal{A}$-coordinates, and because of the additional structure of symplectic local systems, they satisfy additional relations. The following proposition is immediate:
	
	\begin{prop}
		Let $\mathcal L\to S$ be a $\mathcal T$-transverse symplectic local system. Let $\gamma$ be an arc of the triangulation $\mathcal T$ from $p\in P$ to $q\in P$. Then $a_\gamma=\omega(v_q,v_p)$. In particular $a_{\bar\gamma}=-\sigma(a_\gamma)$.
	\end{prop}
	
	\begin{proof}
		By definition of non-commutative $\mathcal A$-coordinates, $v_p\in L_q$ projects to $w_q a_\gamma\in\mathcal L/L_q$, i.e. for some lift $\hat w_q\in A^2$ of $w_q$, $v_p=\hat w_q a_\gamma+v_q r$ for some $r\in A$. Therefore, $\omega(v_q,v_p)=\omega(v_q,\hat w_q a_\gamma+v_q r)=\omega(v_q,\hat w_q) a_\gamma=a_\gamma$.
	\end{proof}
	
	From Proposition~\ref{eq:triangle_relation} follows:
	
	\begin{cor}
		For each oriented triangle $T:=(\gamma_1,\gamma_2,\gamma_3)$ of $\mathcal{T}$, we have $\beta_T:=a_{\gamma_3}a_{\bar\gamma_2}^{-1}a_{\gamma_1}\in A^\sigma$.
		
		If $A$ is Hermitian, the signature of $\beta_T$ agrees with the Kashiwara-Maslov index of $T$.
		
		If $\beta_T\in A^\sigma_+$ for all oriented triangles $T$ of $\T$, then the decorated local system is maximal.
	\end{cor}
	
	\begin{rem}
	    Since these additional relations on $\mathcal{A}$-coordinates involve the structure of $(A,\sigma)$, this is not possible to define a corresponding non-commutative algebra for symplectic local systems as in \cite{BR} for $\GL_2(A)$-local systems.
	\end{rem}
	
	\bibliographystyle{alpha}
	\bibliography{bibl}
\end{document}